\documentclass[twoside,reqno]{amsart}
\usepackage{amsfonts}
\usepackage{graphicx}
\usepackage{amscd}
\usepackage{hyperref}
\usepackage{cite}

\newtheorem{theorem}{Theorem}
\theoremstyle{plain}

\newtheorem{corollary}{Corollary}

\newtheorem{lemma}{Lemma}

\newtheorem{remark}{Remark}

\numberwithin{equation}{section}

\input{tcilatex}

\begin{document}
\title[Generalized mixed equilibrium problems and Hierarchical fixed point
problems]{An explicit iterative method to solve generalized mixed
equilibrium problem, variational inequality problem and hierarchical fixed
point problem for a nearly nonexpansive mapping}
\author{Ibrahim Karahan$^{\ast }$}
\address{$^{\ast }$(Corresponding Author) Department of Mathematics, Faculty
of Science, Erzurum Technical University, Erzurum, 25700, Turkey}
\email{ibrahimkarahan@erzurum.edu.tr}

\begin{abstract}
In this paper, we introduce a new iterative method to find a common solution
of a generalized mixed equilibrium problem, a variational inequality problem
and a hierarchical fixed point problem for a demicontinuous nearly
nonexpansive mapping. We prove that the proposed method converges strongly
to a common solution of above problems under the suitable conditions. It is
also noted that the main theorem is proved without usual demiclosedness
condition. Also, under the appropriate assumptions on the control sequences
and operators, our iterative method can be reduced to recent methods. So,
the results here improve and extend some recent corresponding results given
by many other authors.
\end{abstract}

\maketitle

\section{Introduction}

Let $H$ be a real Hilbert space with inner product $\left\langle \cdot
,\cdot \right\rangle $ and norm $\left\Vert \cdot \right\Vert $, $C$ be a
nonempty closed convex subset of $H$. Let $G:C\times C\rightarrow 
%TCIMACRO{\U{211d} }%
%BeginExpansion
\mathbb{R}
%EndExpansion
$ be a bifunction, $\varphi :C\rightarrow 
%TCIMACRO{\U{211d} }%
%BeginExpansion
\mathbb{R}
%EndExpansion
$ be a function where $%
%TCIMACRO{\U{211d} }%
%BeginExpansion
\mathbb{R}
%EndExpansion
$ is the set of all real numbers and $B$ be a nonlinear mapping. The
generalized mixed equilibrium problem ($GMEP$), is finding a point $x\in C$
such that%
\begin{equation}
G\left( x,y\right) +\varphi \left( y\right) -\varphi \left( x\right)
+\left\langle Bx,y-x\right\rangle \geq 0\text{, }\forall y\in C.  \label{1.1}
\end{equation}
The set of solutions of the problem (\ref{1.1}) is denoted by $GMEP\left(
G,\varphi ,B\right) $. In the problem (\ref{1.1}), if we take $B=0$, then
the generalized mixed equilibrium problem is reduced to mixed equilibrium
problem, denoted by $MEP$, which is to find a point $x\in C$ such that%
\begin{equation*}
G\left( x,y\right) +\varphi \left( y\right) -\varphi \left( x\right) \geq 0%
\text{, }\forall y\in C.
\end{equation*}%
The set of solutions of the mixed equilibrium problem is denoted by $%
MEP(G,\varphi )$. In the case of $\varphi =0$ in the problem (\ref{1.1}),
then the generalized mixed equilibrium problem is reduced to generalized
equilibrium problem, denoted by $GEP$, which is to find a point $x\in C$
such that 
\begin{equation*}
G\left( x,y\right) +\left\langle Bx,y-x\right\rangle \geq 0\text{, }\forall
y\in C
\end{equation*}%
whose set of solutions is denoted by $GEP(G,B)$. If we take $\varphi =0$ and 
$G\left( x,y\right) =0$ for all $x,y\in C$, then the generalized mixed
equilibrium problem is equivalent to find a $x\in C$ such that%
\begin{equation}
\left\langle Bx,y-x\right\rangle \geq 0\text{, }\forall y\in C  \label{1.2}
\end{equation}%
which is called the variational inequality problem, denoted by $VI\left(
C,B\right) $. The solution of $VI\left( C,B\right) $ is denoted by $\Omega $%
, i.e.,%
\begin{equation*}
\Omega =\left\{ x\in C:\left\langle Bx,y-x\right\rangle \geq 0\text{, }%
\forall y\in C\right\} .
\end{equation*}

The generalized mixed equilibrium problem is very general in the sense that
it includes, as special cases, the optimization problem, the variational
inequality problem, the fixed point problem, the nonlinear complementarity
problem, the Nash equilibrium problem in noncooperative games, the vector
optimization problem, the saddle poin problem, the minimization problem and
so forth. Hence, some solution methods have been studied to solve
generalized mixed equilibrium problem by many authors; see, for example \cite%
{PY, LCZ, SK, S, CY, PWK}.

On the other hand, we consider another problem called as hierarchical fixed
point problem. Let $S:C\rightarrow H$ be a mapping. To hierarchically find a
fixed point of a mapping $T$ with respect to another mapping $S$ is to find
an $x^{\ast }\in Fix(T)$ such that%
\begin{equation}
\left\langle x^{\ast }-Sx^{\ast },x-x^{\ast }\right\rangle \geq 0\text{, \ }%
x\in Fix(T),  \label{Hyr}
\end{equation}%
where $Fix\left( T\right) $ is the set of fixed points of $T$, i.e., $%
Fix\left( T\right) =\left\{ x\in C:Tx=x\right\} $. It is known that the
hierarchical fixed point problem (\ref{Hyr}) is related with some monotone
variational inequalities and convex programming problems; see \cite%
{CMMY,YCL,GVC,YC}. Hence, various methods to solve the hierarchical fixed
point problem have been studied by many authors; see\textbf{,} for example%
\textbf{\ }\cite{YCL,cian,marxu,xu,deng} and the references therein.

Now, we give some definitions of nonlinear mappings which are used in the
next sections. Let $T:C\rightarrow H$ be a mapping. If there exits a
constant $L>0$ such that $\left\Vert Tx-Ty\right\Vert \leq L\left\Vert
x-y\right\Vert $ for all $x,y\in C$, then $T$ is called $L$-Lipschitzian. In
particular, if $L\in \lbrack 0,1)$, then $T$ is said to be a contraction; if 
$L=1$, then $T$ is called a nonexpansive mapping. Let fix a sequence $%
\{a_{n}\}$ in $[0,\infty )$ with $a_{n}\rightarrow 0$. If the inequality $%
\left\Vert T^{n}x-T^{n}y\right\Vert \leq \left\Vert x-y\right\Vert +a_{n}$
holds for all $x,y\in C$ and $n\geq 1$, then $T$ is said to be nearly
nonexpansive \cite{AOS1,AOS2} with respect to $\{a_{n}\}$. It is clear that
the class of nearly nonexpansive mappings is a wider class of nonexpansive
mappings. A mapping $F:C\rightarrow H$ is called $\eta $-strongly monotone
operator if there exists a constant $\eta \geq 0$ such that%
\begin{equation*}
\left\langle Fx-Fy,x-y\right\rangle \geq \eta \left\Vert x-y\right\Vert ^{2},%
\text{ }\forall x,y\in C.
\end{equation*}%
In particular, if $\eta =0$, then $F$ is said to be monotone.

Below, we gather some iterative processes which are related with the
problems (\ref{1.1}), (\ref{1.2}) and (\ref{Hyr}).

In 2011, Ceng et al. \cite{Ceng} introduced the following iterative method:%
\begin{equation}
x_{n+1}=P_{C}\left[ \alpha _{n}\rho Vx_{n}+\left( 1-\alpha _{n}\mu F\right)
Tx_{n}\right] ,\forall n\geq 1,  \label{f2}
\end{equation}%
where $F$ is a $L$-Lipschitzian and $\eta $-strongly monotone operator with
constants $L,\eta >0$ and$\ V$ is a $\gamma $-Lipschitzian (possibly
non-self) mapping with constant $\gamma \geq 0$ such that $0<\mu <\frac{%
2\eta }{L^{2}}$ and $0\leq \rho \gamma <1-\sqrt{1-\mu \left( 2\eta -\mu
L^{2}\right) }$. Under the suitable conditions, they proved that the
sequence $\{x_{n}\}$ generated by (\ref{f2}) converges strongly to the
unique solution of the variational inequality%
\begin{equation}
\left\langle \left( \rho V-\mu F\right) x^{\ast },x-x^{\ast }\right\rangle
\leq 0,\text{ }\forall x\in Fix(T).  \label{v2}
\end{equation}%
Recently, motivated by the iteration method (\ref{f2}), Wang and Xu \cite{WX}
introduced the following iterative method to find a solution for a
hierarchical fixed point problem:%
\begin{equation}
\left\{ 
\begin{array}{l}
y_{n}=\beta _{n}Sx_{n}+\left( 1-\beta _{n}\right) x_{n},\text{ \ \ \ \ \ \ \
\ \ \ \ \ \ \ \ \ \ \ \ \ \ \ \ } \\ 
x_{n+1}=P_{C}\left[ \alpha _{n}\rho Vx_{n}+\left( I-\alpha _{n}\mu F\right)
Ty_{n}\right] ,\text{ }\forall n\geq 1,%
\end{array}%
\right.  \label{f3}
\end{equation}%
where $S,T:C\rightarrow C$ are nonexpansive mappings, $V:C\rightarrow H$ is
a $\gamma $-Lipschitzian mapping and $F:C\rightarrow H$ is a $L$%
-Lipschitzian and $\eta $-strongly monotone operator. They proved that under
some approximate assumptions on parameters, the sequence $\{x_{n}\}$
generated by (\ref{f3}) converges strongly to the hierarchical fixed point
of $T$ with respect to the mapping $S$ which is the unique solution of the
variational inequality (\ref{v2}). So, the iterative method (\ref{f3}) is a
generalization form of the various methods of other authors.

Very recently, Bnouhachem and Noor \cite{BN} introduced the following
iterative scheme to approach a common solution of a variational inequality
problem, a generalized equilibrium problem and a hierarchical fixed point
problem:%
\begin{equation}
\left\{ 
\begin{array}{l}
G\left( u_{n},y\right) +\left\langle Bx,y-u_{n}\right\rangle +\frac{1}{r_{n}}%
\left\langle y-u_{n},u_{n}-x_{n}\right\rangle \geq 0\text{, }\forall y\in C
\\ 
z_{n}=P_{C}\left( u_{n}-\lambda _{n}Au_{n}\right) , \\ 
y_{n}=P_{C}\left( \beta _{n}Sx_{n}+\left( 1-\beta _{n}\right) z_{n}\right) ,
\\ 
x_{n+1}=P_{C}\left( \alpha _{n}fx_{n}+\sum_{i=1}^{n}\left( \alpha
_{i-1}-\alpha _{i}\right) V_{i}y_{n}\right) \text{, }\forall n\geq 1%
\end{array}%
\right. ,  \label{f4}
\end{equation}%
where $V_{i}=k_{i}I+\left( 1-k_{i}\right) T_{i}$, $0\leq k_{i}<1$, $\left\{
T_{i}\right\} _{i=1}^{\infty }:C\rightarrow C$ is a countable family of $%
k_{i}$-strict pseudo-contraction mappings, $A$ and $B$ are inverse strongly
monotone mappings. They proved that the sequence $\left\{ x_{n}\right\} $
generated by (\ref{f4}) converges strongly to a point $z\in P_{\Omega \cap
GEP\left( G,B\right) \cap Fix\left( T\right) }f\left( z\right) $ which is
the unique solution of the following variational inequality:%
\begin{equation*}
\left\langle \left( I-f\right) z,x-z\right\rangle \geq 0\text{, }\forall
x\in \Omega \cap GEP\left( G\right) \cap Fix\left( T\right) ,
\end{equation*}%
where $Fix\left( T\right) =\bigcap\nolimits_{i=1}^{\infty }Fix\left(
T_{i}\right) $.

In this paper, motivated and inspired by the above iterative methods, we
introduce an iterative projection method. Also, we prove a strong
convergence theorem to compute an approximate element of the common set of
solutions of a generalized mixed equilibrium problem, a variational
inequality problem and a hierarchically fixed point problem for a nearly
nonexpansive mapping. The proposed method generalizes many known results;
for example, Yao et. al. \cite{YCL}, Marino and Xu \cite{marxu}, Ceng et.
al. \cite{Ceng}, Wang and Xu \cite{WX}, Moudafi \cite{mou}, Xu \cite{xu1},
Tian \cite{tian} and Suzuki \cite{suzuki} and the references therein.

\section{\protect\bigskip Preliminaries}

In this section, we gather some useful lemmas and definitions which we need
for the next section. Throughout this paper, we use $"\rightarrow "$ and $%
"\rightharpoonup "$ for the strong and weak convergence, respectively. Let $%
C $ be a nonempty closed convex subset of a real Hilbert space $H$. It is
known that for any $x\in H$, there exists a unique nearest point in $C$
denoted by $P_{C}x$ such that%
\begin{equation*}
\left\Vert x-P_{C}x\right\Vert =\inf_{y\in C}\left\Vert x-y\right\Vert \text{%
, }\forall x\in H
\end{equation*}%
The mapping $P_{C}:H\rightarrow C$ is called a metric projection. For a
metric projection $P_{C}$, the following inequalities are hold:

\begin{enumerate}
\item $\left\Vert P_{C}x-P_{C}y\right\Vert \leq \left\Vert x-y\right\Vert $, 
$\forall x,y\in H,$

\item $\left\langle x-y,P_{C}x-P_{C}y\right\rangle \geq \left\Vert
P_{C}x-P_{C}y\right\Vert ^{2}$, $\forall x,y\in H,$

\item $\left\langle x-P_{C}x,P_{C}x-y\right\rangle \geq 0$, $\forall x\in
H,y\in C$.
\end{enumerate}

\begin{lemma}
\cite{Ceng}\label{str} Let $V:C\rightarrow H$ be a $\gamma $-Lipschitzian
mapping with a constant $\gamma \geq 0$ and let $F:C\rightarrow H$ be a $L$%
-Lipschitzian and $\eta $-strongly monotone operator with constants $L,\eta
>0$. Then for $0\leq \rho \gamma <\mu \eta ,$%
\begin{equation*}
\left\langle \left( \mu F-\rho V\right) x-\left( \mu F-\rho V\right)
y,x-y\right\rangle \geq \left( \mu \eta -\rho \gamma \right) \left\Vert
x-y\right\Vert ^{2},\text{ }\forall x,y\in C.
\end{equation*}%
That is, $\mu F-\rho V$ is strongly monotone with coefficient $\mu \eta
-\rho \gamma $.
\end{lemma}

\begin{lemma}
\cite{Yamu}\label{cont} Let $C$ be a nonempty subset of a real Hilbert space 
$H.$ Suppose that $\lambda \in \left( 0,1\right) $ and $\mu >0$. Let $%
F:C\rightarrow H$ be a $L$-Lipschitzian and $\eta $-strongly monotone
operator on $C$. Define the mapping $g:C\rightarrow H$ by%
\begin{equation*}
g:=I-\lambda \mu F
\end{equation*}%
Then, $g$ is a contraction that provided $\mu <\frac{2\eta }{L^{2}}$. More
precisely, for $\mu \in \left( 0,\frac{2\eta }{L^{2}}\right) ,$%
\begin{equation*}
\left\Vert g\left( x\right) -g\left( y\right) \right\Vert \leq \left(
1-\lambda \nu \right) \left\Vert x-y\right\Vert \text{, }\forall x,y\in C,
\end{equation*}%
where $\nu =1-\sqrt{1-\mu \left( 2\eta -\mu L^{2}\right) }$.
\end{lemma}

\begin{lemma}
\label{Y}\cite{xu1} Assume that $\left\{ x_{n}\right\} $ is a sequence of
nonnegative real numbers satisfying the conditions%
\begin{equation*}
x_{n+1}\leq \left( 1-\alpha _{n}\right) x_{n}+\alpha _{n}\beta _{n},\text{ }%
\forall n\geq 1
\end{equation*}%
where $\left\{ \alpha _{n}\right\} $ \ and $\left\{ \beta _{n}\right\} $ are
sequences of real numbers such that 
\begin{eqnarray*}
&\text{(i)}&\left\{ \alpha _{n}\right\} \subset \left[ 0,1\right] \text{ and 
}\tsum_{n=1}^{\infty }\alpha _{n}=\infty \text{,\ \ \ \ \ \ \ \ \ \ \ \ \ \
\ \ \ \ \ \ \ \ \ \ \ \ \ \ \ \ \ \ \ \ \ \ \ \ \ \ \ \ \ \ \ \ \ \ \ \ \ }
\\
&\text{(ii)}&\limsup_{n\rightarrow \infty }\beta _{n}\leq 0\text{.}
\end{eqnarray*}%
Then $\lim_{n\rightarrow \infty }x_{n}=0.$
\end{lemma}

For solving an equilibrium problem for a bifunction $G:C\times C\rightarrow 
%TCIMACRO{\U{211d} }%
%BeginExpansion
\mathbb{R}
%EndExpansion
,$ let assume that $G$ satisfies the following conditions:

\begin{enumerate}
\item[(A1)] $G\left( x,x\right) =0,$ $\forall x\in C,$

\item[(A2)] $G$ is monotone, i.e. $G\left( x,y\right) +G\left( y,x\right)
\leq 0,$ $\forall x,y\in C,$

\item[(A3)] $\forall x,y,z\in C,$%
\begin{equation*}
\lim_{t\rightarrow 0^{+}}G\left( tz+\left( 1-t\right) x,y\right) \leq
G\left( x,y\right) ,
\end{equation*}

\item[(A4)] $\forall x\in C,$ $y\longmapsto G\left( x,y\right) $ is convex
and lower semicontinuous.

\item[(B1)] For each $x\in H$ and $r>0$, there exist a bounded subset $%
D_{x}\subseteq C$ and $y_{x}\in C$ such that for any $z\in C\backslash D_{x}$%
\begin{equation*}
G\left( z,y_{x}\right) +\varphi \left( y_{x}\right) -\varphi \left( z\right)
+\frac{1}{r}\left\langle y_{x}-z,z-x\right\rangle <0;
\end{equation*}

\item[(B2)] $C$ is a bounded set.
\end{enumerate}

\begin{lemma}
\label{7}\cite{PY} Let $C$ be a nonempty closed convex subset of a H\.{I}%
lbert space $H$. Let $G:C\times C\rightarrow 
%TCIMACRO{\U{211d} }%
%BeginExpansion
\mathbb{R}
%EndExpansion
$ be a bifunction satisfying (A1)-(A4) and let $\varphi :C\rightarrow 
%TCIMACRO{\U{211d} }%
%BeginExpansion
\mathbb{R}
%EndExpansion
$ be a proper lower semicontinuous and convex function. Assume that either
(B1) or (B2) holds. For $r>0$ and $x\in H,$ define a mapping $%
T_{r}:H\rightarrow C$ as follows:%
\begin{equation*}
T_{r}\left( x\right) =\left\{ z\in C:G\left( z,y\right) +\varphi \left(
y\right) -\varphi \left( z\right) +\frac{1}{r}\left\langle
y-z,z-x\right\rangle \geq 0,\text{ }\forall y\in C\right\}
\end{equation*}%
for all $x\in H.$ Then, the following hold:

\begin{enumerate}
\item[(1)] For each $x\in H$, $T_{r}\left( x\right) \neq \emptyset $

\item[(2)] $T_{r}$ is single valued,

\item[(3)] $T_{r}$ is firmly nonexpansive i.e.%
\begin{equation*}
\left\Vert T_{r}x-T_{r}y\right\Vert ^{2}\leq \left\langle
T_{r}x-T_{r}y,x-y\right\rangle ,\text{ }\forall x,y\in H,
\end{equation*}

\item[(4)] $Fix\left( T_{r}\right) =MEP\left( G\right) ,$

\item[(5)] $MEP\left( G\right) $ is closed and convex.
\end{enumerate}
\end{lemma}

Now, we give the definitions of a demicontinuous mapping, asymptotic radius
and asmptotic center.

Let $C$ be a nonempty subset of a Banach space $X$ and $T:C\rightarrow C$ be
a mapping. For a sequence $\left\{ x_{n}\right\} $ in $C$ which converges
strongly to $x\in X$, if $\left\{ Tx_{n}\right\} $ converges weakly to $Tx$,
then $T$ is called demicontinuous.

Let $C$ be a nonempty closed convex subset of a uniformly convex Banach
space $X$, $\left\{ x_{n}\right\} $ be a bounded sequence in $X$ and $%
r:C\rightarrow \left[ 0,\infty \right) $ be a functional defined by%
\begin{equation*}
r\left( x\right) =\limsup_{n\rightarrow \infty }\left\Vert
x_{n}-x\right\Vert ,\text{ }x\in C.
\end{equation*}%
The infimum of $r\left( \cdot \right) $ over $C$ is called asymptotic radius
of $\left\{ x_{n}\right\} $ with respect to $C$ and is denoted by $r\left(
C,\left\{ x_{n}\right\} \right) $. A point $w\in C$ is said to be an
asymptotic center of the sequence $\left\{ x_{n}\right\} $ with respect to $%
C $ if%
\begin{equation*}
r\left( w\right) =\min \left\{ r\left( x\right) :x\in C\right\} .
\end{equation*}%
The set of all asymptotic centers is denoted by $A\left( C,\left\{
x_{n}\right\} \right) $. Related with these definitions, we will use the
followings in our main results.

\begin{theorem}
\label{C}\cite{AOS2} Let $C$ be a nonempty closed convex subset of a
uniformly convex Banach space $X$ satisfying the Opial condition. If $%
\left\{ x_{n}\right\} $ is a sequence in $C$ such that $x_{n}\rightharpoonup 
$ $w$, then $w$ is the asymptotic center of $\left\{ x_{n}\right\} $ in $C$.
\end{theorem}

\begin{lemma}
\label{b}\cite{sahu1} Let $C$ be a nonempty closed convex subset of a
uniformly convex Banach space $X$ and $T:C\rightarrow C$ be a demicontinuous
nearly Lipschitzian mapping with sequence $\left\{ a_{n},\eta \left(
T^{n}\right) \right\} $ such that $\lim_{n\rightarrow \infty }\eta \left(
T^{n}\right) \leq 1$. If $\left\{ x_{n}\right\} $ is a bounded sequence in $%
C $ such that%
\begin{equation*}
\lim_{m\rightarrow \infty }\left( \lim_{n\rightarrow \infty }\left\Vert
x_{n}-T^{m}x_{n}\right\Vert \right) =0\text{ and }A\left( C,\left\{
x_{n}\right\} \right) =\left\{ w\right\} ,
\end{equation*}%
then $w$ is a fixed point of $T$.
\end{lemma}

\section{Main Results}

Let $C$ be a nonempty closed convex subset of a real Hilbert space $H.$ Let $%
A,B:C\rightarrow H$ be $\alpha ,\theta $-inverse strongly monotone mappings,
respectively. Let $G:C\times C\rightarrow 
%TCIMACRO{\U{211d} }%
%BeginExpansion
\mathbb{R}
%EndExpansion
$ be a bifunction satisfying assumptions (A1)-(A4), $\varphi :C\rightarrow 
%TCIMACRO{\U{211d} }%
%BeginExpansion
\mathbb{R}
%EndExpansion
$ be a lower semicontinuous and convex function, $S:C\rightarrow H$ be a
nonexpansive and $T$ be a demicontinuous nearly nonexpansive mapping with
the sequence $\left\{ a_{n}\right\} $ such that $\tciFourier :=Fix\left(
T\right) \cap \Omega \cap GMEP\left( G,\varphi ,B\right) \neq \emptyset $.
Let $V:C\rightarrow H$ be a $\gamma $-Lipschitzian mapping, $F:C\rightarrow
H $ be a $L$-Lipschitzian and $\eta $-strongly monotone operator such that
these coefficients satisfy $0<\mu <\frac{2\eta }{L^{2}}$, $0\leq \rho \gamma
<\nu $, where $\nu =1-\sqrt{1-\mu \left( 2\eta -\mu L^{2}\right) }$. Assume
that either (B1) or (B2) holds. For an arbitrarily initial value $x_{1},$
define the sequence $\left\{ x_{n}\right\} $ in $C$ generated by%
\begin{equation}
\left\{ 
\begin{array}{l}
G\left( u_{n},y\right) +\varphi \left( y\right) -\varphi \left( u_{n}\right)
+\left\langle Bx_{n},y-u_{n}\right\rangle +\frac{1}{r_{n}}\left\langle
y-u_{n},u_{n}-x_{n}\right\rangle \geq 0\text{, }\forall y\in C\text{\ \ \ \
\ \ } \\ 
z_{n}=P_{C}\left( u_{n}-\lambda _{n}Au_{n}\right) , \\ 
y_{n}=P_{C}\left[ \beta _{n}Sx_{n}+\left( 1-\beta _{n}\right) z_{n}\right] ,%
\text{ } \\ 
x_{n+1}=P_{C}\left[ \alpha _{n}\rho Vx_{n}+\left( I-\alpha _{n}\mu F\right)
T^{n}y_{n}\right] ,\text{ }n\geq 1,%
\end{array}%
\right.  \label{4}
\end{equation}%
where $\left\{ \lambda _{n}\right\} \subset \left( 0,2\alpha \right) $, $%
\left\{ r_{n}\right\} \subset \left( 0,2\theta \right) $, $\left\{ \alpha
_{n}\right\} $ and $\left\{ \beta _{n}\right\} $ are sequences in $\left[ 0,1%
\right] $.

It is known that convergence of a sequence\ depends on the choice of the
control sequences and mappings. So, we consider the following hypotheses on
our control sequences and mappings:%
\begin{eqnarray*}
&(C1)&\lim_{n\rightarrow \infty }\alpha _{n}=0\text{ and }%
\tsum_{n=1}^{\infty }\alpha _{n}=\infty \text{;} \\
&(C2)&\lim_{n\rightarrow \infty }\frac{a_{n}}{\alpha _{n}}=0\text{, }%
\lim_{n\rightarrow \infty }\frac{\beta _{n}}{\alpha _{n}}=0\text{, }%
\lim_{n\rightarrow \infty }\frac{\left\vert \alpha _{n}-\alpha
_{n-1}\right\vert }{\alpha _{n}}=0,\lim_{n\rightarrow \infty }\frac{%
\left\vert \lambda _{n}-\lambda _{n-1}\right\vert }{\alpha _{n}}=0 \\
&&\lim_{n\rightarrow \infty }\frac{\left\vert \beta _{n}-\beta
_{n-1}\right\vert }{\alpha _{n}}=0\text{, and }\lim_{n\rightarrow \infty }%
\frac{\left\vert r_{n}-r_{n-1}\right\vert }{\alpha _{n}}=0\text{;} \\
&(C3)&\lim_{n\rightarrow \infty }\left\Vert T^{n}x-T^{n-1}x\right\Vert =0%
\text{ and }\lim_{n\rightarrow \infty }\frac{\left\Vert
T^{n}x-T^{n-1}x\right\Vert }{\alpha _{n}}=0,\forall x\in C\text{. \ \ }
\end{eqnarray*}%
Before giving the main theorem, we have to prove the following lemmas.

\begin{lemma}
\label{L1} Assume that the conditions $(C1)$ and $(C2)$\ hold. Let $p\in
\tciFourier $. Then, the sequences $\left\{ x_{n}\right\} ,\left\{
y_{n}\right\} ,\left\{ z_{n}\right\} $ and $\left\{ u_{n}\right\} $
generated by (\ref{4}) are bounded.
\end{lemma}

\begin{proof}
It is easy to see that the mappings $I-r_{n}B$ and $I-\lambda _{n}A$ are
nonexpansive. Indeed, since $\left\{ r_{n}\right\} \subset \left( 0,2\theta
\right) $, we have%
\begin{eqnarray*}
\left\Vert \left( I-r_{n}B\right) x-\left( I-r_{n}B\right) y\right\Vert ^{2}
&=&\left\Vert x-y\right\Vert ^{2}-2r_{n}\left\langle x-y,Bx-By\right\rangle
+r_{n}^{2}\left\Vert Bx-By\right\Vert ^{2} \\
&\leq &\left\Vert x-y\right\Vert ^{2}-2r_{n}\theta \left\Vert
Bx-By\right\Vert +r_{n}^{2}\left\Vert Bx-By\right\Vert ^{2} \\
&\leq &\left\Vert x-y\right\Vert ^{2}-r_{n}\left( 2\theta -r_{n}\right)
\left\Vert Bx-By\right\Vert ^{2} \\
&\leq &\left\Vert x-y\right\Vert ^{2}.
\end{eqnarray*}%
Similarly, since $\left\{ \lambda _{n}\right\} \subset \left( 0,2\alpha
\right) $,\ we have%
\begin{eqnarray*}
\left\Vert \left( I-\lambda _{n}A\right) x-\left( I-\lambda _{n}A\right)
y\right\Vert ^{2} &=&\left\Vert x-y\right\Vert ^{2}-2\lambda
_{n}\left\langle x-y,Ax-Ay\right\rangle +\lambda _{n}^{2}\left\Vert
Ax-Ay\right\Vert ^{2} \\
&\leq &\left\Vert x-y\right\Vert ^{2}-\lambda _{n}\left( 2\alpha -\lambda
_{n}\right) \left\Vert Ax-Ay\right\Vert ^{2} \\
&\leq &\left\Vert x-y\right\Vert ^{2}.
\end{eqnarray*}%
It follows from Lemma \ref{7} that $u_{n}=T_{r_{n}}\left(
x_{n}-r_{n}Bx_{n}\right) $. Let $p\in \tciFourier $. So, we get $%
p=T_{r_{n}}\left( p-r_{n}Bp\right) $ and%
\begin{eqnarray}
\left\Vert u_{n}-p\right\Vert ^{2} &=&\left\Vert T_{r_{n}}\left(
x_{n}-r_{n}Bx_{n}\right) -T_{r_{n}}\left( p-r_{n}Bp\right) \right\Vert ^{2} 
\notag \\
&\leq &\left\Vert \left( x_{n}-r_{n}Bx_{n}\right) -\left( p-r_{n}Bp\right)
\right\Vert ^{2}  \notag \\
&\leq &\left\Vert x_{n}-p\right\Vert ^{2}-r_{n}\left( 2\theta -r_{n}\right)
\left\Vert Bx_{n}-Bp\right\Vert ^{2}  \notag \\
&\leq &\left\Vert x_{n}-p\right\Vert ^{2}.  \label{A1}
\end{eqnarray}%
By using (\ref{A1}), we obtain%
\begin{eqnarray}
\left\Vert z_{n}-p\right\Vert ^{2} &=&\left\Vert P_{C}\left( u_{n}-\lambda
_{n}Au_{n}\right) -P_{C}\left( p-\lambda _{n}Ap\right) \right\Vert ^{2} 
\notag \\
&\leq &\left\Vert u_{n}-p-\lambda _{n}\left( Au_{n}-Ap\right) \right\Vert
^{2}  \notag \\
&\leq &\left\Vert u_{n}-p\right\Vert ^{2}-\lambda _{n}\left( 2\alpha
-\lambda _{n}\right) \left\Vert Au_{n}-Ap\right\Vert ^{2}  \notag \\
&\leq &\left\Vert u_{n}-p\right\Vert ^{2}  \notag \\
&\leq &\left\Vert x_{n}-p\right\Vert ^{2}\text{.}  \label{A2}
\end{eqnarray}%
So, from (\ref{A2}), we have%
\begin{eqnarray}
\left\Vert y_{n}-p\right\Vert &=&\left\Vert P_{C}\left[ \beta
_{n}Sx_{n}+\left( 1-\beta _{n}\right) x_{n}\right] -P_{C}p\right\Vert  \notag
\\
&\leq &\left\Vert \beta _{n}Sx_{n}+\left( 1-\beta _{n}\right)
z_{n}-p\right\Vert  \notag \\
&\leq &\left( 1-\beta _{n}\right) \left\Vert z_{n}-p\right\Vert +\beta
_{n}\left\Vert Sx_{n}-p\right\Vert  \notag \\
&\leq &\left( 1-\beta _{n}\right) \left\Vert x_{n}-p\right\Vert +\beta
_{n}\left\Vert Sx_{n}-Sp\right\Vert +\beta _{n}\left\Vert Sp-p\right\Vert 
\notag \\
&\leq &\left\Vert x_{n}-p\right\Vert +\beta _{n}\left\Vert Sp-p\right\Vert .
\label{2}
\end{eqnarray}%
Since $\lim_{n\rightarrow \infty }\frac{\beta _{n}}{\alpha _{n}}=0$, without
loss of generality, we can assume that $\beta _{n}\leq \alpha _{n}$, for all 
$n\geq 1.$ This gives us $\lim_{n\rightarrow \infty }\beta _{n}=0$. Let $%
t_{n}=\alpha _{n}\rho Vx_{n}+\left( I-\alpha _{n}\mu F\right) T^{n}y_{n}$.
Then, we have%
\begin{eqnarray}
\left\Vert x_{n+1}-p\right\Vert &=&\left\Vert P_{C}t_{n}-P_{C}p\right\Vert 
\notag \\
&\leq &\left\Vert t_{n}-p\right\Vert  \notag \\
&=&\left\Vert \alpha _{n}\rho Vx_{n}+\left( I-\alpha _{n}\mu F\right)
T^{n}y_{n}-p\right\Vert  \notag \\
&\leq &\alpha _{n}\left\Vert \rho Vx_{n}-\mu Fp\right\Vert +\left\Vert
\left( I-\alpha _{n}\mu F\right) T^{n}y_{n}-\left( I-\alpha _{n}\mu F\right)
T^{n}p\right\Vert  \notag \\
&\leq &\alpha _{n}\rho \gamma \left\Vert x_{n}-p\right\Vert +\alpha
_{n}\left\Vert \rho Vp-\mu Fp\right\Vert  \notag \\
&&+\left( 1-\alpha _{n}\nu \right) \left( \left\Vert y_{n}-p\right\Vert
+a_{n}\right) .  \label{3}
\end{eqnarray}%
So, it follows from (\ref{2}) and (\ref{3}) that%
\begin{eqnarray}
\left\Vert x_{n+1}-p\right\Vert &\leq &\alpha _{n}\rho \gamma \left\Vert
x_{n}-p\right\Vert +\alpha _{n}\left\Vert \rho Vp-\mu Fp\right\Vert  \notag
\\
&&+\left( 1-\alpha _{n}\nu \right) \left( \left\Vert x_{n}-p\right\Vert
+\beta _{n}\left\Vert Sp-p\right\Vert +a_{n}\right)  \notag \\
&\leq &\left( 1-\alpha _{n}\left( \nu -\rho \gamma \right) \right)
\left\Vert x_{n}-p\right\Vert +\alpha _{n}\left( \left\Vert \rho Vp-\mu
Fp\right\Vert +\left\Vert Sp-p\right\Vert +\frac{a_{n}}{\alpha _{n}}\right) 
\notag \\
&\leq &\left( 1-\alpha _{n}\left( \nu -\rho \gamma \right) \right)
\left\Vert x_{n}-p\right\Vert  \notag \\
&&+\alpha _{n}\left( \nu -\rho \gamma \right) \left[ \frac{1}{\left( \nu
-\rho \gamma \right) }\left( \left\Vert \rho Vp-\mu Fp\right\Vert
+\left\Vert Sp-p\right\Vert +\frac{a_{n}}{\alpha _{n}}\right) \right] .
\label{e3}
\end{eqnarray}%
From condition $(C2)$, there exists a constant $M_{1}>0$ such that%
\begin{equation*}
\left\Vert \rho Vp-\mu Fp\right\Vert +\left\Vert Sp-p\right\Vert +\frac{a_{n}%
}{\alpha _{n}}\leq M_{1}\text{, }\forall n\geq 1\text{.}
\end{equation*}%
Hence, from (\ref{e3}) we get%
\begin{equation*}
\left\Vert x_{n+1}-p\right\Vert \leq \left( 1-\alpha _{n}\left( \nu -\rho
\gamma \right) \right) \left\Vert x_{n}-p\right\Vert +\alpha _{n}\left( \nu
-\rho \gamma \right) \frac{M_{1}}{\left( \nu -\rho \gamma \right) }.
\end{equation*}%
By induction, we obtain%
\begin{equation*}
\left\Vert x_{n+1}-p\right\Vert \leq \max \left\{ \left\Vert
x_{1}-p\right\Vert ,\frac{M}{\left( \nu -\rho \gamma \right) }\right\} .
\end{equation*}%
Therefore, we obtain that $\left\{ x_{n}\right\} $ is bounded. So, the
sequences $\left\{ y_{n}\right\} $,$\left\{ z_{n}\right\} $ and $\left\{
u_{n}\right\} $ are bounded.
\end{proof}

\begin{lemma}
\label{L2} Assume that $(C1)$-$(C3)$ hold and $p\in \tciFourier $. Let $%
\left\{ x_{n}\right\} $ be the sequence generated by (\ref{4}). Then, the
followings are hold:

\begin{enumerate}
\item[(i)] $\lim_{n\rightarrow \infty }\left\Vert x_{n+1}-x_{n}\right\Vert
=0.$

\item[(ii)] The weak $w$-limit set $w_{w}\left( x_{n}\right) \subset
Fix\left( T\right) .$
\end{enumerate}
\end{lemma}

\begin{proof}
(i) Since the metric projection $P_{C}$ and the mapping $\left( I-\lambda
_{n}A\right) $ are nonexpansive, we have%
\begin{eqnarray}
\left\Vert z_{n}-z_{n-1}\right\Vert &=&\left\Vert P_{C}\left( u_{n}-\lambda
_{n}Au_{n}\right) -P_{C}\left( u_{n-1}-\lambda _{n-1}Au_{n-1}\right)
\right\Vert  \notag \\
&\leq &\left\Vert \left( u_{n}-\lambda _{n}Au_{n}\right) -\left(
u_{n-1}-\lambda _{n-1}Au_{n-1}\right) \right\Vert  \notag \\
&=&\left\Vert u_{n}-u_{n-1}-\lambda _{n}\left( Au_{n}-Au_{n-1}\right)
-\left( \lambda _{n}-\lambda _{n-1}\right) Au_{n-1}\right\Vert  \notag \\
&\leq &\left\Vert u_{n}-u_{n-1}-\lambda _{n}\left( Au_{n}-Au_{n-1}\right)
\right\Vert +\left\vert \lambda _{n}-\lambda _{n-1}\right\vert \left\Vert
Au_{n-1}\right\Vert  \notag \\
&\leq &\left\Vert u_{n}-u_{n-1}\right\Vert +\left\vert \lambda _{n}-\lambda
_{n-1}\right\vert \left\Vert Au_{n-1}\right\Vert ,  \label{A3}
\end{eqnarray}%
and so%
\begin{eqnarray}
\left\Vert y_{n}-y_{n-1}\right\Vert &=&\left\Vert P_{C}\left[ \beta
_{n}Sx_{n}+\left( 1-\beta _{n}\right) z_{n}\right] \right.  \notag \\
&&\left. -P_{C}\left[ \beta _{n-1}Sx_{n-1}-\left( 1-\beta _{n-1}\right)
z_{n-1}\right] \right\Vert  \notag \\
&\leq &\left\Vert \beta _{n}Sx_{n}+\left( 1-\beta _{n}\right) z_{n}-\beta
_{n-1}Sx_{n-1}+\left( 1-\beta _{n-1}\right) z_{n-1}\right\Vert  \notag \\
&\leq &\left\Vert \beta _{n}\left( Sx_{n}-Sx_{n-1}\right) +\left( \beta
_{n}-\beta _{n-1}\right) Sx_{n-1}\right.  \notag \\
&&\left. +\left( 1-\beta _{n}\right) \left( z_{n}-z_{n-1}\right) +\left(
\beta _{n-1}-\beta _{n}\right) z_{n-1}\right\Vert  \notag \\
&\leq &\beta _{n}\left\Vert x_{n}-x_{n-1}\right\Vert +\left( 1-\beta
_{n}\right) \left\Vert z_{n}-z_{n-1}\right\Vert  \notag \\
&&+\left\vert \beta _{n}-\beta _{n-1}\right\vert \left( \left\Vert
Sx_{n-1}\right\Vert +\left\Vert z_{n-1}\right\Vert \right)  \notag \\
&\leq &\beta _{n}\left\Vert x_{n}-x_{n-1}\right\Vert +\left( 1-\beta
_{n}\right) \left[ \left\Vert u_{n}-u_{n-1}\right\Vert \right.  \notag \\
&&\left. +\left\vert \lambda _{n}-\lambda _{n-1}\right\vert \left\Vert
Au_{n-1}\right\Vert \right]  \notag \\
&&+\left\vert \beta _{n}-\beta _{n-1}\right\vert \left( \left\Vert
Sx_{n-1}\right\Vert +\left\Vert z_{n-1}\right\Vert \right) .  \label{A5}
\end{eqnarray}%
On the other side, since $u_{n}=T_{r_{n}}\left( x_{n}-r_{n}Bx_{n}\right) $
and $u_{n-1}=T_{r_{n-1}}\left( x_{n-1}-r_{n-1}Bx_{n-1}\right) $, we get%
\begin{equation}
\begin{array}{l}
G\left( u_{n},y\right) +\varphi \left( y\right) -\varphi \left( u_{n}\right)
+\left\langle Bx_{n},y-u_{n}\right\rangle \\ 
+\frac{1}{r_{n}}\left\langle y-u_{n},u_{n}-x_{n}\right\rangle \geq 0\text{, }%
\forall y\in C%
\end{array}
\label{A6}
\end{equation}%
and%
\begin{equation}
\begin{array}{l}
G\left( u_{n-1},y\right) +\varphi \left( y\right) -\varphi \left(
u_{n-1}\right) +\left\langle Bx_{n-1},y-u_{n-1}\right\rangle \\ 
+\frac{1}{r_{n-1}}\left\langle y-u_{n-1},u_{n-1}-x_{n-1}\right\rangle \geq 0%
\text{, }\forall y\in C.%
\end{array}
\label{A7}
\end{equation}%
If we take $y=u_{n-1}$ in (\ref{A6}) and $y=u_{n}$ in (\ref{A7}), then we
have%
\begin{equation}
\begin{array}{l}
G\left( u_{n},u_{n-1}\right) +\varphi \left( u_{n-1}\right) -\varphi \left(
u_{n}\right) +\left\langle Bx_{n},u_{n-1}-u_{n}\right\rangle \\ 
+\frac{1}{r_{n}}\left\langle u_{n-1}-u_{n},u_{n}-x_{n}\right\rangle \geq 0%
\text{,}%
\end{array}
\label{A8}
\end{equation}%
and%
\begin{equation}
\begin{array}{l}
G\left( u_{n-1},u_{n}\right) +\varphi \left( u_{n}\right) -\varphi \left(
u_{n-1}\right) +\left\langle Bx_{n-1},u_{n}-u_{n-1}\right\rangle \\ 
+\frac{1}{r_{n-1}}\left\langle u_{n}-u_{n-1},u_{n-1}-x_{n-1}\right\rangle
\geq 0.%
\end{array}
\label{A9}
\end{equation}%
By using the monotonicity of the bifunction $G$ and the inequalitites (\ref%
{A8}) and (\ref{A9}), we obtain%
\begin{equation*}
\left\langle Bx_{n-1}-Bx_{n},u_{n}-u_{n-1}\right\rangle +\left\langle
u_{n}-u_{n-1},\frac{u_{n-1}-x_{n-1}}{r_{n-1}}-\frac{u_{n}-x_{n}}{r_{n}}%
\right\rangle \geq 0.
\end{equation*}

It follows from the last inequality that%
\begin{eqnarray}
0 &\leq &\left\langle u_{n}-u_{n-1},r_{n}\left( Bx_{n-1}-Bx_{n}\right) +%
\frac{r_{n}}{r_{n-1}}\left( u_{n-1}-x_{n-1}\right) -\left(
u_{n}-x_{n}\right) \right\rangle  \notag \\
&=&\left\langle u_{n-1}-u_{n},u_{n}-u_{n-1}+\left( 1-\frac{r_{n}}{r_{n-1}}%
\right) u_{n-1}\right.  \notag \\
&&\left. +\left( x_{n-1}-r_{n}Bx_{n-1}\right) -\left(
x_{n}-r_{n}Bx_{n}\right) -x_{n-1}+\frac{r_{n}}{r_{n-1}}x_{n-1}\right\rangle 
\notag \\
&=&\left\langle u_{n-1}-u_{n},\left( 1-\frac{r_{n}}{r_{n-1}}\right)
u_{n-1}+\left( x_{n-1}-r_{n}Bx_{n-1}\right) \right.  \notag \\
&&\left. -\left( x_{n}-r_{n}Bx_{n}\right) -x_{n-1}+\frac{r_{n}}{r_{n-1}}%
x_{n-1}\right\rangle -\left\Vert u_{n}-u_{n-1}\right\Vert ^{2}  \notag \\
&=&\left\langle u_{n-1}-u_{n},\left( 1-\frac{r_{n}}{r_{n-1}}\right) \left(
u_{n-1}-x_{n-1}\right) \right.  \notag \\
&&\left. +\left( x_{n-1}-r_{n}Bx_{n-1}\right) -\left(
x_{n}-r_{n}Bx_{n}\right) \right\rangle -\left\Vert u_{n}-u_{n-1}\right\Vert
^{2}  \notag \\
&\leq &\left\Vert u_{n-1}-u_{n}\right\Vert \left\{ \left\vert 1-\frac{r_{n}}{%
r_{n-1}}\right\vert \left\Vert u_{n-1}-x_{n-1}\right\Vert \right.  \notag \\
&&\left. +\left\Vert \left( x_{n-1}-r_{n}Bx_{n-1}\right) -\left(
x_{n}-r_{n}Bx_{n}\right) \right\Vert \right\} -\left\Vert
u_{n}-u_{n-1}\right\Vert ^{2}  \notag \\
&\leq &\left\Vert u_{n-1}-u_{n}\right\Vert \left\{ \left\vert 1-\frac{r_{n}}{%
r_{n-1}}\right\vert \left\Vert u_{n-1}-x_{n-1}\right\Vert \right.  \notag \\
&&\left. +\left\Vert x_{n-1}-x_{n}\right\Vert \right\} -\left\Vert
u_{n}-u_{n-1}\right\Vert ^{2}.  \label{A10}
\end{eqnarray}%
From (\ref{A10}), we have%
\begin{equation*}
\left\Vert u_{n-1}-u_{n}\right\Vert \leq \left\vert 1-\frac{r_{n}}{r_{n-1}}%
\right\vert \left\Vert u_{n-1}-x_{n-1}\right\Vert +\left\Vert
x_{n-1}-x_{n}\right\Vert .
\end{equation*}%
Without loss of generality, we can assume that there exists a real number $%
\mu $ such that $r_{n}>\mu >0$ for all positive integers $n$. Then, we have%
\begin{equation}
\left\Vert u_{n-1}-u_{n}\right\Vert \leq \left\Vert x_{n-1}-x_{n}\right\Vert
+\frac{1}{\mu }\left\vert r_{n-1}-r_{n}\right\vert \left\Vert
u_{n-1}-x_{n-1}\right\Vert .  \label{A11}
\end{equation}%
(\ref{A5}) and (\ref{A11}) imply that%
\begin{eqnarray*}
\left\Vert y_{n}-y_{n-1}\right\Vert &\leq &\beta _{n}\left\Vert
x_{n}-x_{n-1}\right\Vert \\
&&+\left( 1-\beta _{n}\right) \left[ \left\Vert x_{n-1}-x_{n}\right\Vert +%
\frac{1}{\mu }\left\vert r_{n-1}-r_{n}\right\vert \left\Vert
u_{n-1}-x_{n-1}\right\Vert \right. \\
&&\left. +\left\vert \lambda _{n}-\lambda _{n-1}\right\vert \left\Vert
Au_{n-1}\right\Vert \right] +\left\vert \beta _{n}-\beta _{n-1}\right\vert
\left( \left\Vert Sx_{n-1}\right\Vert +\left\Vert z_{n-1}\right\Vert \right)
\\
&=&\left\Vert x_{n}-x_{n-1}\right\Vert +\left( 1-\beta _{n}\right) \left[ 
\frac{1}{\mu }\left\vert r_{n-1}-r_{n}\right\vert \left\Vert
u_{n-1}-x_{n-1}\right\Vert \right. \\
&&\left. +\left\vert \lambda _{n}-\lambda _{n-1}\right\vert \left\Vert
Au_{n-1}\right\Vert \right] +\left\vert \beta _{n}-\beta _{n-1}\right\vert
\left( \left\Vert Sx_{n-1}\right\Vert +\left\Vert z_{n-1}\right\Vert \right)
.
\end{eqnarray*}%
Hence, we get%
\begin{eqnarray}
\left\Vert x_{n+1}-x_{n}\right\Vert &=&\left\Vert
P_{C}t_{n}-P_{C}t_{n-1}\right\Vert  \notag \\
&\leq &\left\Vert t_{n}-t_{n-1}\right\Vert  \notag \\
&=&\left\Vert \alpha _{n}\rho Vx_{n}+\left( I-\alpha _{n}\mu F\right)
T^{n}y_{n}\right.  \notag \\
&&\left. -\alpha _{n-1}\rho Vx_{n-1}+\left( I-\alpha _{n-1}\mu F\right)
T^{n-1}y_{n-1}\right\Vert  \notag \\
&\leq &\left\Vert \alpha _{n}\rho V\left( x_{n}-x_{n-1}\right) +\left(
\alpha _{n}-\alpha _{n-1}\right) \rho Vx_{n-1}\right.  \notag \\
&&+\left( I-\alpha _{n}\mu F\right) T^{n}y_{n}-\left( I-\alpha _{n}\mu
F\right) T^{n}y_{n-1}  \notag \\
&&+T^{n}y_{n-1}-T^{n-1}y_{n-1}  \notag \\
&&\left. +\alpha _{n-1}\mu FT^{n-1}y_{n-1}-\alpha _{n}\mu
FT^{n}y_{n-1}\right\Vert  \notag \\
&\leq &\alpha _{n}\rho \gamma \left\Vert x_{n}-x_{n-1}\right\Vert +\gamma
\left\vert \alpha _{n}-\alpha _{n-1}\right\vert \left\Vert
Vx_{n-1}\right\Vert  \notag \\
&&+\left( 1-\alpha _{n}\nu \right) \left\Vert
T^{n}y_{n}-T^{n}y_{n-1}\right\Vert +\left\Vert
T^{n}y_{n-1}-T^{n-1}y_{n-1}\right\Vert  \notag \\
&&+\mu \left\Vert \alpha _{n-1}FT^{n-1}y_{n-1}-\alpha
_{n}FT^{n}y_{n-1}\right\Vert  \notag \\
&\leq &\alpha _{n}\rho \gamma \left\Vert x_{n}-x_{n-1}\right\Vert +\gamma
\left\vert \alpha _{n}-\alpha _{n-1}\right\vert \left\Vert
Vx_{n-1}\right\Vert  \notag \\
&&+\left( 1-\alpha _{n}\nu \right) \left[ \left\Vert
y_{n}-y_{n-1}\right\Vert +a_{n}\right] +\left\Vert
T^{n}y_{n-1}-T^{n-1}y_{n-1}\right\Vert  \notag \\
&&+\mu \left\Vert \alpha _{n-1}\left( FT^{n-1}y_{n-1}-FT^{n}y_{n-1}\right)
-\left( \alpha _{n}-\alpha _{n-1}\right) FT^{n}y_{n-1}\right\Vert  \notag \\
&\leq &\alpha _{n}\rho \gamma \left\Vert x_{n}-x_{n-1}\right\Vert +\gamma
\left\vert \alpha _{n}-\alpha _{n-1}\right\vert \left\Vert
Vx_{n-1}\right\Vert  \notag \\
&&+\left( 1-\alpha _{n}\nu \right) \left\{ \left\Vert
x_{n}-x_{n-1}\right\Vert \right.  \notag \\
&&+\left( 1-\beta _{n}\right) \left[ \frac{1}{\mu }\left\vert
r_{n-1}-r_{n}\right\vert \left\Vert u_{n-1}-x_{n-1}\right\Vert +\left\vert
\lambda _{n}-\lambda _{n-1}\right\vert \left\Vert Au_{n-1}\right\Vert \right]
\notag \\
&&\left. +\left\vert \beta _{n}-\beta _{n-1}\right\vert \left( \left\Vert
Sx_{n-1}\right\Vert +\left\Vert z_{n-1}\right\Vert \right) \right\}  \notag
\\
&&+\left( 1-\alpha _{n}\nu \right) a_{n}+\mathfrak{D}_{B}\left(
T_{n},T^{n-1}\right)  \notag \\
&&+\mu \alpha _{n-1}L\left\Vert T^{n}y_{n-1}-T^{n-1}y_{n-1}\right\Vert
+\left\vert \alpha _{n}-\alpha _{n-1}\right\vert \left\Vert
FT^{n}y_{n-1}\right\Vert  \notag \\
&\leq &\left( 1-\alpha _{n}\left( \nu -\rho \gamma \right) \right)
\left\Vert x_{n}-x_{n-1}\right\Vert  \notag \\
&&+\left\vert \alpha _{n}-\alpha _{n-1}\right\vert \left( \gamma \left\Vert
Vx_{n-1}\right\Vert +\left\Vert FT^{n}y_{n-1}\right\Vert \right)  \notag \\
&&+\left( 1+\mu \alpha _{n-1}L\right) \left\Vert
T^{n}y_{n-1}-T^{n-1}y_{n-1}\right\Vert +a_{n}  \notag \\
&&+\frac{1}{\mu }\left\vert r_{n-1}-r_{n}\right\vert \left\Vert
u_{n-1}-x_{n-1}\right\Vert +\left\vert \lambda _{n}-\lambda
_{n-1}\right\vert \left\Vert Au_{n-1}\right\Vert  \notag \\
&&+\left\vert \beta _{n}-\beta _{n-1}\right\vert \left( \left\Vert
Sx_{n-1}\right\Vert +\left\Vert z_{n-1}\right\Vert \right)  \notag \\
&\leq &\left( 1-\alpha _{n}\left( \nu -\rho \gamma \right) \right)
\left\Vert x_{n}-x_{n-1}\right\Vert +\left( 1+\mu \alpha _{n-1}L\right)
\left\Vert T^{n}y_{n-1}-T^{n-1}y_{n-1}\right\Vert  \notag \\
&&+M_{2}\left( \left\vert \alpha _{n}-\alpha _{n-1}\right\vert +\frac{1}{\mu 
}\left\vert r_{n-1}-r_{n}\right\vert \right.  \notag \\
&&\left. +\left\vert \lambda _{n}-\lambda _{n-1}\right\vert +\left\vert
\beta _{n}-\beta _{n-1}\right\vert \right) +a_{n}  \label{A12}
\end{eqnarray}%
where%
\begin{eqnarray*}
M_{2} &=&\max \left\{ \sup_{n\geq 1}\left( \gamma \left\Vert
Vx_{n-1}\right\Vert +\left\Vert FT^{n}y_{n-1}\right\Vert \right) \text{, }%
\sup_{n\geq 1}\left\Vert u_{n-1}-x_{n-1}\right\Vert \text{,}\right. \\
&&\left. \sup_{n\geq 1}\left\Vert Au_{n-1}\right\Vert \text{, }\sup_{n\geq
1}\left( \left\Vert Sx_{n-1}\right\Vert +\left\Vert z_{n-1}\right\Vert
\right) \right\} .
\end{eqnarray*}%
Therefore, we obtain%
\begin{equation}
\left\Vert x_{n+1}-x_{n}\right\Vert \leq \left( 1-\alpha _{n}\left( \nu
-\rho \gamma \right) \right) \left\Vert x_{n}-x_{n-1}\right\Vert +\alpha
_{n}\left( \nu -\rho \gamma \right) \delta _{n},  \label{A13}
\end{equation}%
where%
\begin{eqnarray*}
\delta _{n} &=&\frac{1}{\left( \nu -\rho \gamma \right) }\left[ \left( 1+\mu
\alpha _{n-1}L\right) \frac{\left\Vert
T^{n}y_{n-1}-T^{n-1}y_{n-1}\right\Vert }{\alpha _{n}}\right. \\
&&+\frac{a_{n}}{\alpha _{n}}+M_{2}\left( \frac{\left\vert \alpha _{n}-\alpha
_{n-1}\right\vert }{\alpha _{n}}+\frac{1}{\mu }\frac{\left\vert
r_{n-1}-r_{n}\right\vert }{\alpha _{n}}\right. \\
&&+\left. \left. \frac{\left\vert \lambda _{n}-\lambda _{n-1}\right\vert }{%
\alpha _{n}}+\frac{\left\vert \beta _{n}-\beta _{n-1}\right\vert }{\alpha
_{n}}\right) \right] .
\end{eqnarray*}%
By using conditions $(C2)$ and $(C3)$, we get%
\begin{equation}
\limsup_{n\rightarrow \infty }\delta _{n}\leq 0.  \label{A14}
\end{equation}%
So, it follows from (\ref{A13}), (\ref{A14}) and Lemma \ref{Y} that%
\begin{equation}
\lim_{n\rightarrow \infty }\left\Vert x_{n+1}-x_{n}\right\Vert =0.
\label{A16}
\end{equation}%
(ii) Now, we show that the weak $w$-limit set of $\left\{ x_{n}\right\} $ is
a subset of the set of fixed points of $T.$ To show that, we need to show $%
\lim_{n\rightarrow \infty }\left\Vert u_{n}-x_{n}\right\Vert =0$. Let $p\in
\tciFourier $. Then, by using (\ref{A1}) and (\ref{A2}), we get%
\begin{eqnarray}
\left\Vert x_{n+1}-p\right\Vert ^{2} &\leq &\left\Vert t_{n}-p\right\Vert
^{2}  \notag \\
&=&\left\Vert \alpha _{n}\rho Vx_{n}+\left( I-\alpha _{n}\mu F\right)
T^{n}y_{n}-p\right\Vert ^{2}  \notag \\
&=&\left\Vert \alpha _{n}\rho Vx_{n}-\alpha _{n}\mu Fp+\left( I-\alpha
_{n}\mu F\right) T^{n}y_{n}-\left( I-\alpha _{n}\mu F\right)
T^{n}p\right\Vert ^{2}  \notag \\
&\leq &\alpha _{n}\left\Vert \rho Vx_{n}-\mu Fp\right\Vert ^{2}+\left(
1-\alpha _{n}\nu \right) \left( \left\Vert y_{n}-p\right\Vert +a_{n}\right)
^{2}  \notag \\
&=&\alpha _{n}\left\Vert \rho Vx_{n}-\mu Fp\right\Vert ^{2}  \notag \\
&&+\left( 1-\alpha _{n}\nu \right) \left( \left\Vert y_{n}-p\right\Vert
^{2}+2a_{n}\left\Vert y_{n}-p\right\Vert +a_{n}^{2}\right)  \notag \\
&=&\alpha _{n}\left\Vert \rho Vx_{n}-\mu Fp\right\Vert ^{2}+\left( 1-\alpha
_{n}\nu \right) \left\Vert y_{n}-p\right\Vert ^{2}  \notag \\
&&+2\left( 1-\alpha _{n}\nu \right) a_{n}\left\Vert y_{n}-p\right\Vert
+\left( 1-\alpha _{n}\nu \right) a_{n}^{2}  \notag \\
&\leq &\alpha _{n}\left\Vert \rho Vx_{n}-\mu Fp\right\Vert ^{2}+\left(
1-\alpha _{n}\nu \right) \left[ \beta _{n}\left\Vert Sx_{n}-p\right\Vert
^{2}\right.  \notag \\
&&+\left. \left( 1-\beta _{n}\right) \left\Vert z_{n}-p\right\Vert ^{2} 
\right] +2\left( 1-\alpha _{n}\nu \right) a_{n}\left\Vert y_{n}-p\right\Vert
+\left( 1-\alpha _{n}\nu \right) a_{n}^{2}  \notag \\
&=&\alpha _{n}\left\Vert \rho Vx_{n}-\mu Fp\right\Vert ^{2}+\left( 1-\alpha
_{n}\nu \right) \beta _{n}\left\Vert Sx_{n}-p\right\Vert ^{2}  \notag \\
&&+\left( 1-\alpha _{n}\nu \right) \left( 1-\beta _{n}\right) \left\Vert
z_{n}-p\right\Vert ^{2}  \notag \\
&&+2\left( 1-\alpha _{n}\nu \right) a_{n}\left\Vert y_{n}-p\right\Vert
+\left( 1-\alpha _{n}\nu \right) a_{n}^{2}  \notag \\
&\leq &\alpha _{n}\left\Vert \rho Vx_{n}-\mu Fp\right\Vert ^{2}+\left(
1-\alpha _{n}\nu \right) \beta _{n}\left\Vert Sx_{n}-p\right\Vert ^{2} 
\notag \\
&&+\left( 1-\alpha _{n}\nu \right) \left( 1-\beta _{n}\right) \left[
\left\Vert x_{n}-p\right\Vert ^{2}-r_{n}\left( 2\theta -r_{n}\right)
\left\Vert Bx_{n}-Bp\right\Vert ^{2}\right.  \notag \\
&&\left. -\lambda _{n}\left( 2\alpha -\lambda _{n}\right) \left\Vert
Au_{n}-Ap\right\Vert ^{2}\right]  \notag \\
&&+2\left( 1-\alpha _{n}\nu \right) a_{n}\left\Vert y_{n}-p\right\Vert
+\left( 1-\alpha _{n}\nu \right) a_{n}^{2}  \notag \\
&\leq &\alpha _{n}\left\Vert \rho Vx_{n}-\mu Fp\right\Vert ^{2}+\beta
_{n}\left\Vert Sx_{n}-p\right\Vert ^{2}+\left\Vert x_{n}-p\right\Vert ^{2} 
\notag \\
&&-\left( 1-\alpha _{n}\nu \right) \left( 1-\beta _{n}\right) \left[
r_{n}\left( 2\theta -r_{n}\right) \left\Vert Bx_{n}-Bp\right\Vert ^{2}\right.
\notag \\
&&+\left. \lambda _{n}\left( 2\alpha -\lambda _{n}\right) \left\Vert
Au_{n}-Ap\right\Vert ^{2}\right]  \notag \\
&&+2\left( 1-\alpha _{n}\nu \right) a_{n}\left\Vert y_{n}-p\right\Vert
+\left( 1-\alpha _{n}\nu \right) a_{n}^{2}  \label{A15}
\end{eqnarray}%
The inequality (\ref{A15}) implies that%
\begin{eqnarray*}
&&\left( 1-\alpha _{n}\nu \right) \left( 1-\beta _{n}\right) \left\{
r_{n}\left( 2\theta -r_{n}\right) \left\Vert Bx_{n}-Bp\right\Vert
^{2}+\lambda _{n}\left( 2\alpha -\lambda _{n}\right) \left\Vert
Au_{n}-Ap\right\Vert ^{2}\right\} \\
&\leq &\alpha _{n}\left\Vert \rho Vx_{n}-\mu Fp\right\Vert ^{2}+\beta
_{n}\left\Vert Sx_{n}-p\right\Vert ^{2}+\left\Vert x_{n}-p\right\Vert
^{2}-\left\Vert x_{n+1}-p\right\Vert ^{2} \\
&&+2\left( 1-\alpha _{n}\nu \right) a_{n}\left\Vert y_{n}-p\right\Vert
+\left( 1-\alpha _{n}\nu \right) a_{n}^{2} \\
&\leq &\alpha _{n}\left\Vert \rho Vx_{n}-\mu Fp\right\Vert ^{2}+\beta
_{n}\left\Vert Sx_{n}-p\right\Vert ^{2}+\left( \left\Vert x_{n}-p\right\Vert
+\left\Vert x_{n+1}-p\right\Vert \right) \left\Vert x_{n+1}-p\right\Vert \\
&&+2\left( 1-\alpha _{n}\nu \right) a_{n}\left\Vert y_{n}-p\right\Vert
+\left( 1-\alpha _{n}\nu \right) a_{n}^{2}.
\end{eqnarray*}%
So, it follows from (\ref{A16}), conditions $(C1)$ and $(C2)$ that $%
\lim_{n\rightarrow \infty }\left\Vert Bx_{n}-Bp\right\Vert =0$ and $%
\lim_{n\rightarrow \infty }\left\Vert Au_{n}-Ap\right\Vert =0$. On the other
side, we know from Lemma \ref{7} that $T_{r_{n}}$ is firmly nonexpansive
mapping, we get%
\begin{eqnarray*}
\left\Vert u_{n}-p\right\Vert ^{2} &=&\left\Vert T_{r_{n}}\left(
x_{n}-r_{n}Bx_{n}\right) -T_{r_{n}}\left( p-r_{n}Bp\right) \right\Vert ^{2}
\\
&\leq &\left\langle u_{n}-p,\left( x_{n}-r_{n}Bx_{n}\right) -\left(
p-r_{n}Bp\right) \right\rangle \\
&=&\frac{1}{2}\left\{ \left\Vert u_{n}-p\right\Vert ^{2}+\left\Vert \left(
x_{n}-r_{n}Bx_{n}\right) -\left( p-r_{n}Bp\right) \right\Vert ^{2}\right. \\
&&-\left. \left\Vert u_{n}-p-\left[ \left( x_{n}-r_{n}Bx_{n}\right) -\left(
p-r_{n}Bp\right) \right] \right\Vert ^{2}\right\} .
\end{eqnarray*}%
which implies that 
\begin{eqnarray}
\left\Vert u_{n}-p\right\Vert ^{2} &\leq &\left\Vert \left(
x_{n}-r_{n}Bx_{n}\right) -\left( p-r_{n}Bp\right) \right\Vert ^{2}  \notag \\
&&-\left\Vert u_{n}-x_{n}+r_{n}\left( Bx_{n}-Bp\right) \right\Vert ^{2} 
\notag \\
&\leq &\left\Vert x_{n}-p\right\Vert ^{2}-\left\Vert u_{n}-x_{n}+r_{n}\left(
Bx_{n}-Bp\right) \right\Vert ^{2}  \notag \\
&\leq &\left\Vert x_{n}-p\right\Vert ^{2}-\left\Vert u_{n}-x_{n}\right\Vert
^{2}  \notag \\
&&+2r_{n}\left\Vert u_{n}-x_{n}\right\Vert \left\Vert Bx_{n}-Bp\right\Vert .
\label{A17}
\end{eqnarray}%
Then, from (\ref{A2}), (\ref{A15}) and (\ref{A17}), we have%
\begin{eqnarray*}
\left\Vert x_{n+1}-p\right\Vert ^{2} &\leq &\alpha _{n}\left\Vert \rho
Vx_{n}-\mu Fp\right\Vert ^{2}+\left( 1-\alpha _{n}\nu \right) \left[ \beta
_{n}\left\Vert Sx_{n}-p\right\Vert ^{2}\right. \\
&&+\left. \left( 1-\beta _{n}\right) \left\Vert z_{n}-p\right\Vert ^{2} 
\right] +2\left( 1-\alpha _{n}\nu \right) a_{n}\left\Vert y_{n}-p\right\Vert
+\left( 1-\alpha _{n}\nu \right) a_{n}^{2} \\
&\leq &\alpha _{n}\left\Vert \rho Vx_{n}-\mu Fp\right\Vert ^{2}+\left(
1-\alpha _{n}\nu \right) \left[ \beta _{n}\left\Vert Sx_{n}-p\right\Vert
^{2}\right. \\
&&+\left. \left( 1-\beta _{n}\right) \left\Vert u_{n}-p\right\Vert ^{2} 
\right] +2\left( 1-\alpha _{n}\nu \right) a_{n}\left\Vert y_{n}-p\right\Vert
+\left( 1-\alpha _{n}\nu \right) a_{n}^{2} \\
&\leq &\alpha _{n}\left\Vert \rho Vx_{n}-\mu Fp\right\Vert ^{2}+\left(
1-\alpha _{n}\nu \right) \left[ \beta _{n}\left\Vert Sx_{n}-p\right\Vert
^{2}\right. \\
&&+\left( 1-\beta _{n}\right) \left( \left\Vert x_{n}-p\right\Vert
^{2}-\left\Vert u_{n}-x_{n}\right\Vert ^{2}\right. \\
&&+\left. \left. 2r_{n}\left\Vert u_{n}-x_{n}\right\Vert \left\Vert
Bx_{n}-Bp\right\Vert \right) \right] \\
&&+2\left( 1-\alpha _{n}\nu \right) a_{n}\left\Vert y_{n}-p\right\Vert
+\left( 1-\alpha _{n}\nu \right) a_{n}^{2} \\
&\leq &\alpha _{n}\left\Vert \rho Vx_{n}-\mu Fp\right\Vert ^{2}+\beta
_{n}\left\Vert Sx_{n}-p\right\Vert ^{2}+\left\Vert x_{n}-p\right\Vert ^{2} \\
&&-\left( 1-\alpha _{n}\nu \right) \left( 1-\beta _{n}\right) \left\Vert
u_{n}-x_{n}\right\Vert ^{2}+2r_{n}\left\Vert u_{n}-x_{n}\right\Vert
\left\Vert Bx_{n}-Bp\right\Vert \\
&&+2\left( 1-\alpha _{n}\nu \right) a_{n}\left\Vert y_{n}-p\right\Vert
+\left( 1-\alpha _{n}\nu \right) a_{n}^{2}.
\end{eqnarray*}%
From the last inequality, we obtain%
\begin{eqnarray*}
&&\left( 1-\alpha _{n}\nu \right) \left( 1-\beta _{n}\right) \left\Vert
u_{n}-x_{n}\right\Vert ^{2} \\
&\leq &\alpha _{n}\left\Vert \rho Vx_{n}-\mu Fp\right\Vert ^{2}+\beta
_{n}\left\Vert Sx_{n}-p\right\Vert ^{2} \\
&&+\left\Vert x_{n}-p\right\Vert ^{2}-\left\Vert x_{n+1}-p\right\Vert
^{2}+2r_{n}\left\Vert u_{n}-x_{n}\right\Vert \left\Vert Bx_{n}-Bp\right\Vert
\\
&&+2\left( 1-\alpha _{n}\nu \right) a_{n}\left\Vert y_{n}-p\right\Vert
+\left( 1-\alpha _{n}\nu \right) a_{n}^{2} \\
&\leq &\alpha _{n}\left\Vert \rho Vx_{n}-\mu Fp\right\Vert ^{2}+\beta
_{n}\left\Vert Sx_{n}-p\right\Vert ^{2} \\
&&+\left( \left\Vert x_{n}-p\right\Vert +\left\Vert x_{n+1}-p\right\Vert
\right) \left\Vert x_{n+1}-x_{n}\right\Vert \\
&&+2r_{n}\left\Vert u_{n}-x_{n}\right\Vert \left\Vert Bx_{n}-Bp\right\Vert \\
&&+2\left( 1-\alpha _{n}\nu \right) a_{n}\left\Vert y_{n}-p\right\Vert
+\left( 1-\alpha _{n}\nu \right) a_{n}^{2}.
\end{eqnarray*}%
Since $\lim_{n\rightarrow \infty }\left\Vert Bx_{n}-Bp\right\Vert =0$ and $%
\left\{ \left\Vert y_{n}-p\right\Vert \right\} $ is a bounded sequence, it
follows from (\ref{A16}) and conditions $(C1)$, $(C2)$ that%
\begin{equation}
\lim_{n\rightarrow \infty }\left\Vert u_{n}-x_{n}\right\Vert =0.  \label{A18}
\end{equation}%
On the other hand, from the property (2) of the metric projection, we write%
\begin{eqnarray*}
\left\Vert z_{n}-p\right\Vert ^{2} &=&\left\Vert P_{C}\left( u_{n}-\lambda
_{n}Au_{n}\right) -P_{C}\left( p-\lambda _{n}Ap\right) \right\Vert ^{2} \\
&\leq &\left\langle z_{n}-p,\left( u_{n}-\lambda _{n}Au_{n}\right) -\left(
p-\lambda _{n}Ap\right) \right\rangle \\
&=&\frac{1}{2}\left\{ \left\Vert z_{n}-p\right\Vert ^{2}+\left\Vert
u_{n}-p\left( Au_{n}-Ap\right) \right\Vert ^{2}\right. \\
&&-\left. \left\Vert u_{n}-p-\lambda _{n}\left( Au_{n}-Ap\right) -\left(
z_{n}-p\right) \right\Vert ^{2}\right\} \\
&\leq &\frac{1}{2}\left\{ \left\Vert z_{n}-p\right\Vert ^{2}+\left\Vert
u_{n}-p\right\Vert ^{2}\right. \\
&&-\left. \left\Vert u_{n}-z_{n}-\lambda _{n}\left( Au_{n}-Ap\right)
\right\Vert ^{2}\right\} \\
&\leq &\frac{1}{2}\left\{ \left\Vert z_{n}-p\right\Vert ^{2}+\left\Vert
u_{n}-p\right\Vert ^{2}\right. \\
&&-\left. \left\Vert u_{n}-z_{n}\right\Vert ^{2}+2\lambda _{n}\left\langle
u_{n}-z_{n},Au_{n}-Ap\right\rangle \right\} \\
&\leq &\frac{1}{2}\left\{ \left\Vert z_{n}-p\right\Vert ^{2}+\left\Vert
u_{n}-p\right\Vert ^{2}-\left\Vert u_{n}-z_{n}\right\Vert ^{2}\right. \\
&&+\left. 2\lambda _{n}\left\Vert u_{n}-z_{n}\right\Vert \left\Vert
Au_{n}-Ap\right\Vert \right\} .
\end{eqnarray*}%
Hence, we obtain%
\begin{eqnarray}
\left\Vert z_{n}-p\right\Vert ^{2} &\leq &\left\Vert u_{n}-p\right\Vert
^{2}-\left\Vert u_{n}-z_{n}\right\Vert ^{2}  \notag \\
&&+2\lambda _{n}\left\Vert u_{n}-z_{n}\right\Vert \left\Vert
Au_{n}-Ap\right\Vert  \notag \\
&\leq &\left\Vert x_{n}-p\right\Vert ^{2}-\left\Vert u_{n}-z_{n}\right\Vert
^{2}  \notag \\
&&+2\lambda _{n}\left\Vert u_{n}-z_{n}\right\Vert \left\Vert
Au_{n}-Ap\right\Vert .  \label{A19}
\end{eqnarray}%
From (\ref{A15}) and (\ref{A19}), we get%
\begin{eqnarray*}
\left\Vert x_{n+1}-p\right\Vert ^{2} &\leq &\alpha _{n}\left\Vert \rho
Vx_{n}-\mu Fp\right\Vert ^{2}+\left( 1-\alpha _{n}\nu \right) \left[ \beta
_{n}\left\Vert Sx_{n}-p\right\Vert ^{2}\right. \\
&&+\left. \left( 1-\beta _{n}\right) \left\Vert z_{n}-p\right\Vert ^{2} 
\right] +2\left( 1-\alpha _{n}\nu \right) a_{n}\left\Vert y_{n}-p\right\Vert
+\left( 1-\alpha _{n}\nu \right) a_{n}^{2} \\
&\leq &\alpha _{n}\left\Vert \rho Vx_{n}-\mu Fp\right\Vert ^{2}+\left(
1-\alpha _{n}\nu \right) \left[ \beta _{n}\left\Vert Sx_{n}-p\right\Vert
^{2}\right. \\
&&+\left( 1-\beta _{n}\right) \left( \left\Vert x_{n}-p\right\Vert
^{2}-\left\Vert u_{n}-z_{n}\right\Vert ^{2}\right. \\
&&+\left. \left. 2\lambda _{n}\left\Vert u_{n}-z_{n}\right\Vert \left\Vert
Au_{n}-Ap\right\Vert \right) \right] \\
&&+2\left( 1-\alpha _{n}\nu \right) a_{n}\left\Vert y_{n}-p\right\Vert
+\left( 1-\alpha _{n}\nu \right) a_{n}^{2} \\
&\leq &\alpha _{n}\left\Vert \rho Vx_{n}-\mu Fp\right\Vert ^{2}+\beta
_{n}\left\Vert Sx_{n}-p\right\Vert ^{2}+\left\Vert x_{n}-p\right\Vert ^{2} \\
&&-\left( 1-\alpha _{n}\nu \right) \beta _{n}\left\Vert
u_{n}-z_{n}\right\Vert ^{2}+2\lambda _{n}\left\Vert u_{n}-z_{n}\right\Vert
\left\Vert Au_{n}-Ap\right\Vert \\
&&+2\left( 1-\alpha _{n}\nu \right) a_{n}\left\Vert y_{n}-p\right\Vert
+\left( 1-\alpha _{n}\nu \right) a_{n}^{2}.
\end{eqnarray*}%
So, we have%
\begin{eqnarray*}
\left( 1-\alpha _{n}\nu \right) \beta _{n}\left\Vert u_{n}-z_{n}\right\Vert
^{2} &\leq &\alpha _{n}\left\Vert \rho Vx_{n}-\mu Fp\right\Vert ^{2}+\beta
_{n}\left\Vert Sx_{n}-p\right\Vert ^{2} \\
&&+\left\Vert x_{n}-p\right\Vert ^{2}-\left\Vert x_{n+1}-p\right\Vert ^{2} \\
&&+2\lambda _{n}\left\Vert u_{n}-z_{n}\right\Vert \left\Vert
Au_{n}-Ap\right\Vert \\
&&+2\left( 1-\alpha _{n}\nu \right) a_{n}\left\Vert y_{n}-p\right\Vert
+\left( 1-\alpha _{n}\nu \right) a_{n}^{2} \\
&\leq &\alpha _{n}\left\Vert \rho Vx_{n}-\mu Fp\right\Vert ^{2}+\beta
_{n}\left\Vert Sx_{n}-p\right\Vert ^{2} \\
&&+\left( \left\Vert x_{n}-p\right\Vert +\left\Vert x_{n+1}-p\right\Vert
\right) \left\Vert x_{n+1}-x_{n}\right\Vert \\
&&+2\lambda _{n}\left\Vert u_{n}-z_{n}\right\Vert \left\Vert
Au_{n}-Ap\right\Vert \\
&&+2\left( 1-\alpha _{n}\nu \right) a_{n}\left\Vert y_{n}-p\right\Vert
+\left( 1-\alpha _{n}\nu \right) a_{n}^{2}.
\end{eqnarray*}%
Since $\lim_{n\rightarrow \infty }\left\Vert Au_{n}-Ap\right\Vert =0$ and $%
\left\{ \left\Vert y_{n}-p\right\Vert \right\} $ is a bounded sequence, it
follows from (\ref{A16}) and conditions $(C1)$ and $(C2)$ that%
\begin{equation}
\lim_{n\rightarrow \infty }\left\Vert u_{n}-z_{n}\right\Vert =0.  \label{A20}
\end{equation}%
Also, by using (\ref{A18}) and (\ref{A20}), we get%
\begin{equation}
\lim_{n\rightarrow \infty }\left\Vert x_{n}-z_{n}\right\Vert =0.  \label{A24}
\end{equation}%
On the other side, we have%
\begin{eqnarray*}
\left\Vert x_{n}-y_{n}\right\Vert &\leq &\left\Vert x_{n}-u_{n}\right\Vert
+\left\Vert u_{n}-z_{n}\right\Vert +\left\Vert z_{n}-y_{n}\right\Vert \\
&=&\left\Vert x_{n}-u_{n}\right\Vert +\left\Vert u_{n}-z_{n}\right\Vert
+\beta _{n}\left( Sx_{n}-z_{n}\right) .
\end{eqnarray*}%
Since $\lim_{n\rightarrow \infty }\beta _{n}=0$, again from (\ref{A18}) and (%
\ref{A20}), we obtain%
\begin{equation}
\lim_{n\rightarrow \infty }\left\Vert x_{n}-y_{n}\right\Vert =0.  \label{A21}
\end{equation}%
Now, we show that $\lim_{n\rightarrow \infty }\left\Vert
x_{n}-Tx_{n}\right\Vert =0$. Before that, we need to show that $%
\lim_{m\rightarrow \infty }\left( \lim_{n\rightarrow \infty }\left\Vert
x_{n}-T^{m}x_{n}\right\Vert \right) =0$. For $n\geq m\geq 1$, we get%
\begin{eqnarray}
\left\Vert T^{n}y_{n}-T^{m}x_{n}\right\Vert &\leq &\left\Vert
T^{n}y_{n}-T^{n-1}y_{n}\right\Vert +\left\Vert
T^{n-1}y_{n}-T^{n-2}y_{n}\right\Vert  \notag \\
&&+\cdots +\left\Vert T^{m}y_{n}-T^{m}x_{n}\right\Vert  \notag \\
&\leq &\left\Vert T^{n}y_{n}-T^{n-1}y_{n}\right\Vert +\left\Vert
T^{n-1}y_{n}-T^{n-2}y_{n}\right\Vert  \notag \\
&&+\cdots +\left\Vert y_{n}-x_{n}\right\Vert +a_{m}\text{,}  \label{d1}
\end{eqnarray}%
and so%
\begin{eqnarray}
\left\Vert x_{n+1}-T^{m}x_{n}\right\Vert &=&\left\Vert
P_{C}t_{n}-P_{C}T^{m}x_{n}\right\Vert  \notag \\
&\leq &\left\Vert \alpha _{n}\rho Vx_{n}+\left( I-\alpha _{n}\mu F\right)
T^{n}y_{n}-T^{m}x_{n}\right\Vert  \notag \\
&\leq &\alpha _{n}\left\Vert \rho Vx_{n}-\mu FT^{n}y_{n}\right\Vert
+\left\Vert T^{n}y_{n}-T^{m}x_{n}\right\Vert  \notag \\
&\leq &\alpha _{n}\left\Vert \rho Vx_{n}-\mu FT^{n}y_{n}\right\Vert
+\left\Vert T^{n}y_{n}-T^{n-1}y_{n}\right\Vert  \notag \\
&&+\left\Vert T^{n-1}y_{n}-T^{n-2}y_{n}\right\Vert +\cdots +\left\Vert
y_{n}-x_{n}\right\Vert +a_{m}\text{.}  \label{d2}
\end{eqnarray}%
Hence, we obtain from (\ref{d1}) and (\ref{d2}) 
\begin{eqnarray}
\left\Vert x_{n}-T^{m}x_{n}\right\Vert &\leq &\left\Vert
x_{n}-x_{n+1}\right\Vert +\left\Vert x_{n+1}-T^{m}x_{n}\right\Vert  \notag \\
&\leq &\left\Vert x_{n}-x_{n+1}\right\Vert +\alpha _{n}\left\Vert \rho
Vx_{n}-\mu FT^{n}y_{n}\right\Vert  \notag \\
&&+\left\Vert T^{n}y_{n}-T^{n-1}y_{n}\right\Vert +\left\Vert
T^{n-1}y_{n}-T^{n-2}y_{n}\right\Vert  \notag \\
&&+\cdots +\left\Vert y_{n}-x_{n}\right\Vert +a_{m}  \notag
\end{eqnarray}%
Since $\left\Vert \rho Vx_{n}-\mu FT^{n}y_{n}\right\Vert $ is bounded and $%
a_{n}\rightarrow 0$,\ it follows from (\ref{A16}\textbf{), }(\ref{A21}%
\textbf{), }conditions $(C1)$\textbf{\ }and $(C3)$ that 
\begin{equation}
\lim_{m\rightarrow \infty }\left( \lim_{n\rightarrow \infty }\left\Vert
x_{n}-T^{m}x_{n}\right\Vert \right) =0\text{.}  \label{A30}
\end{equation}%
By using (\ref{A30}\textbf{) }and condition $(C3)$, we have 
\begin{equation*}
\left\Vert x_{n}-Tx_{n}\right\Vert \leq \left\Vert
x_{n}-T^{m}x_{n}\right\Vert +\left\Vert T^{m}x_{n}-Tx_{n}\right\Vert
\rightarrow 0\text{, as }n,m\rightarrow \infty .
\end{equation*}%
Since $\left\{ x_{n}\right\} $ is bounded, there exists a weak convergent
subsequence $\left\{ x_{n_{k}}\right\} $ of $\left\{ x_{n}\right\} $. Let $%
x_{n_{k}}\rightharpoonup w$ as $k\rightarrow \infty $. Then, Opial's
condition guarantee that the weakly subsequential limit of $\left\{
x_{n}\right\} $ is unique. Hence, this implies that $x_{n}\rightharpoonup w$%
, as\textbf{\ }$n\rightarrow \infty $.\ So, it follows from (\ref{A30}%
\textbf{), }Theorem \ref{C} and Lemma \ref{b} that\textbf{\ }$w\in Fix\left(
T\right) $\textbf{.} Therefore, $w_{w}\left( x_{n}\right) \subset Fix\left(
T\right) $.
\end{proof}

\begin{theorem}
\label{X} Assume that $(C1)$-$(C3)$\ hold. Then the sequence $\left\{
x_{n}\right\} $ generated by (\ref{4}) converges strongly to $x^{\ast }\in
\tciFourier $, which is the unique solution of the variational inequality%
\begin{equation}
\left\langle \left( \rho V-\mu F\right) x^{\ast },x-x^{\ast }\right\rangle
\leq 0,\text{ }\forall x\in \tciFourier .  \label{1}
\end{equation}
\end{theorem}

\begin{proof}
From Lemma \ref{str}, since the operator $\mu F-\rho V$ is $\left( \mu \eta
-\rho \gamma \right) $-strongly monotone we get the uniqueness of the
solution of the variational inequality (\ref{1}). Let denote this solution
by $x^{\ast }\in \tciFourier $.

Now, we divide our proof into three steps.

\textbf{Step 1.} From Lemma \ref{L1}, since $\left\{ x_{n}\right\} $ is a
bounded sequence, there exists an element $w$ such that $x_{n}%
\rightharpoonup w$. Now, we show that $w\in \tciFourier =Fix\left( T\right)
\cap \Omega \cap GMEP\left( G,\varphi ,B\right) .$ Firstly, it follows from
Lemma \ref{L2} (ii) that $w\in Fix\left( T\right) $. Secondly, we show that $%
w\in GMEP\left( G,\varphi ,B\right) $. Since $u_{n}=T_{r_{n}}\left(
x_{n}-r_{n}Bx_{n}\right) ,$ we have%
\begin{equation*}
G\left( u_{n},y\right) +\varphi \left( y\right) -\varphi \left( u_{n}\right)
+\left\langle Bx_{n},y-u_{n}\right\rangle +\frac{1}{r_{n}}\left\langle
y-u_{n},u_{n}-x_{n}\right\rangle \geq 0\text{, }\forall y\in C\text{.}
\end{equation*}%
which implies that%
\begin{equation*}
\varphi \left( y\right) -\varphi \left( u_{n}\right) +\left\langle
Bx_{n},y-u_{n}\right\rangle +\frac{1}{r_{n}}\left\langle
y-u_{n},u_{n}-x_{n}\right\rangle \geq G\left( y,u_{n}\right) \text{, }%
\forall y\in C,
\end{equation*}%
and hence%
\begin{equation}
\varphi \left( y\right) -\varphi \left( u_{n_{k}}\right) +\left\langle
Bx_{n_{k}},y-u_{n_{k}}\right\rangle +\left\langle y-u_{n_{k}},\frac{%
u_{n_{k}}-x_{n_{k}}}{r_{n_{k}}}\right\rangle \geq G\left( y,u_{n_{k}}\right) 
\text{, }\forall y\in C.  \label{A26}
\end{equation}%
Let $y\in C$ and $y_{t}=ty+\left( 1-t\right) w$, for $t\in \left( 0,1\right] 
$. Then, $y_{t}\in C$. From (\ref{A26}), we obtain%
\begin{eqnarray}
\left\langle By_{t},y_{t}-u_{n_{k}}\right\rangle &\geq &\left\langle
By_{t},y_{t}-u_{n_{k}}\right\rangle -\varphi \left( y_{t}\right) +\varphi
\left( u_{n_{k}}\right) -\left\langle Bx_{n_{k}},y_{t}-u_{n_{k}}\right\rangle
\notag \\
&&-\left\langle y_{t}-u_{n_{k}},\frac{u_{n_{k}}-x_{n_{k}}}{r_{n_{k}}}%
\right\rangle +G\left( y_{t},u_{n_{k}}\right)  \notag \\
&=&\left\langle By_{t}-Bx_{n_{k}},y_{t}-u_{n_{k}}\right\rangle +\left\langle
Bu_{n_{k}}-Bx_{n_{k}},y_{t}-u_{n_{k}}\right\rangle -\varphi \left(
y_{t}\right)  \notag \\
&&+\varphi \left( u_{n_{k}}\right) -\left\langle y_{t}-u_{n_{k}},\frac{%
u_{n_{k}}-x_{n_{k}}}{r_{n_{k}}}\right\rangle +G\left( y_{t},u_{n_{k}}\right)
.  \label{A27}
\end{eqnarray}%
On the other hand, since $B$ is Lipschitz continuous, using (\ref{A18}) we
obtain $\lim_{k\rightarrow \infty }\left\Vert
Bu_{n_{k}}-Bx_{n_{k}}\right\Vert =0$. So, it follows from (\ref{A27}), $%
u_{n_{k}}\rightharpoonup w$ and the monotonicity of $B$ that%
\begin{equation}
\left\langle By_{t},y_{t}-w\right\rangle \geq G\left( y_{t},w\right)
-\varphi \left( y_{t}\right) +\varphi \left( w\right) .  \label{A28}
\end{equation}%
By using the inequality (\ref{A28}) and assumptions (A1)-(A4), we get%
\begin{eqnarray*}
0 &=&G\left( y_{t},y_{t}\right) +\varphi \left( y_{t}\right) -\varphi \left(
y_{t}\right) \\
&\leq &tG\left( y_{t},y\right) +\left( 1-t\right) G\left( y_{t},w\right)
+t\varphi \left( y\right) +\left( 1-t\right) \varphi \left( w\right)
-\varphi \left( y_{t}\right) \\
&=&t\left[ G\left( y_{t},y\right) +\varphi \left( y\right) -\varphi \left(
y_{t}\right) \right] +\left( 1-t\right) \left[ G\left( y_{t},w\right)
+\varphi \left( w\right) -\varphi \left( y_{t}\right) \right] \\
&\leq &t\left[ G\left( y_{t},y\right) +\varphi \left( y\right) -\varphi
\left( y_{t}\right) \right] +\left( 1-t\right) \left\langle
By_{t},y_{t}-w\right\rangle \\
&=&t\left[ G\left( y_{t},y\right) +\varphi \left( y\right) -\varphi \left(
y_{t}\right) \right] +\left( 1-t\right) t\left\langle
By_{t},y-w\right\rangle .
\end{eqnarray*}%
The last inequality implies that%
\begin{equation*}
G\left( y_{t},y\right) +\varphi \left( y\right) -\varphi \left( y_{t}\right)
+\left( 1-t\right) \left\langle By_{t},y-w\right\rangle \geq 0.
\end{equation*}%
If we take limit as $t\rightarrow 0^{+}$ for all $y\in C$, we get%
\begin{equation*}
G\left( w,y\right) +\varphi \left( y\right) -\varphi \left( w\right)
+\left\langle Bw,y-w\right\rangle \geq 0\text{, }\forall y\in C.
\end{equation*}%
Hence, we have $w\in GMEP\left( G,\varphi ,B\right) $. Finally, we show that 
$w\in \Omega $. Let $N_{C}v$ be the normal cone to $C$ at $v\in C$, i.e.,%
\begin{equation*}
N_{C}v=\left\{ w\in H:\left\langle v-u,w\right\rangle \geq 0,\forall u\in
C\right\} .
\end{equation*}%
Let $K$ be a mapping defined by%
\begin{equation*}
Kv=\left\{ 
\begin{array}{ll}
Av+N_{C}v & ,\text{ }v\in C, \\ 
\emptyset & ,\text{ }v\notin C.%
\end{array}%
\right.
\end{equation*}%
Then, it is known that $K$ is maximal monotone mapping. Let $\left(
v,u\right) \in G\left( K\right) .$ From the definition of the the mapping $K$%
, since $u-Av\in N_{C}v$ and $z_{n}\in C,$ we get%
\begin{equation}
\left\langle v-z_{n},u-Av\right\rangle \geq 0.  \label{A23}
\end{equation}%
Also, by using the definition of $z_{n},$ we get%
\begin{equation*}
\left\langle v-z_{n},z_{n}-u_{n}-\lambda _{n}Au_{n}\right\rangle \geq 0
\end{equation*}%
and so,%
\begin{equation*}
\left\langle v-z_{n},\frac{z_{n}-u_{n}}{\lambda _{n}}+Au_{n}\right\rangle
\geq 0.
\end{equation*}%
Hence, from (\ref{A23}), we obtain%
\begin{eqnarray}
\left\langle v-z_{n_{i}},u\right\rangle &\geq &\left\langle
v-z_{n_{i}},Av\right\rangle  \notag \\
&\geq &\left\langle v-z_{n_{i}},Av\right\rangle -\left\langle v-z_{n_{i}},%
\frac{z_{n_{i}}-u_{n_{i}}}{\lambda _{n_{i}}}+Au_{n_{i}}\right\rangle  \notag
\\
&=&\left\langle v-z_{n_{i}},Av-Au_{n_{i}}-\frac{z_{n_{i}}-u_{n_{i}}}{\lambda
_{n_{i}}}\right\rangle  \notag \\
&=&\left\langle v-z_{n_{i}},Av-Az_{n_{i}}\right\rangle +\left\langle
v-z_{n_{i}},Az_{n_{i}}-Au_{n_{i}}\right\rangle  \notag \\
&&-\left\langle v-z_{n_{i}},\frac{z_{n_{i}}-u_{n_{i}}}{\lambda _{n_{i}}}%
\right\rangle  \notag \\
&\geq &\left\langle v-z_{n_{i}},Az_{n_{i}}-Au_{n_{i}}\right\rangle
-\left\langle v-z_{n_{i}},\frac{z_{n_{i}}-u_{n_{i}}}{\lambda _{n_{i}}}%
\right\rangle .  \label{A25}
\end{eqnarray}%
So, it follows from (\ref{A18}), (\ref{A20}) and (\ref{A24}) that $%
u_{n_{i}}\rightharpoonup w$ and $z_{n_{i}}\rightharpoonup w$ for $%
i\rightarrow \infty .$ So, from (\ref{A25}), we have%
\begin{equation*}
\left\langle v-w,u\right\rangle \geq 0.
\end{equation*}%
Since $K$ is maximal monotone, we have $w\in K^{-1}0$ and hence $w\in \Omega 
$. Thus, we obtain $w\in \tciFourier =Fix\left( T\right) \cap \Omega \cap
GMEP\left( G\right) $.

\textbf{Step 2.} In this step, we show that $\limsup_{n\rightarrow \infty
}\left\langle \left( \rho V-\mu F\right) x^{\ast },x_{n}-x^{\ast
}\right\rangle \leq 0$, where $x^{\ast }$\ is the unique solution of the
variational inequality (\ref{1}). Since the sequence $\left\{ x_{n}\right\} $
is bounded, it has a weak convergent subsequence $\left\{ x_{n_{k}}\right\} $
such that%
\begin{equation*}
\limsup_{n\rightarrow \infty }\left\langle \left( \rho V-\mu F\right)
x^{\ast },x_{n}-x^{\ast }\right\rangle =\limsup_{k\rightarrow \infty
}\left\langle \left( \rho V-\mu F\right) x^{\ast },x_{n_{k}}-x^{\ast
}\right\rangle .
\end{equation*}%
Let $x_{n_{k}}\rightharpoonup w$, as $k\rightarrow \infty $. Since the Opial
condition guarantee that $x_{n}\rightharpoonup w$, we know from Step 1 that $%
w\in \tciFourier $. Hence%
\begin{equation*}
\limsup_{n\rightarrow \infty }\left\langle \left( \rho V-\mu F\right)
x^{\ast },x_{n}-x^{\ast }\right\rangle =\left\langle \left( \rho V-\mu
F\right) x^{\ast },w-x^{\ast }\right\rangle \leq 0.
\end{equation*}%
\textbf{Step 3.} Finally, we show that the sequence $\left\{ x_{n}\right\} $
generated by (\ref{4}) converges strongly to the point $x^{\ast }$ which\ is
the unique solution of the variational inequality (\ref{1}). From the
definition of the iterative sequence $\left\{ x_{n}\right\} $, we get%
\begin{eqnarray}
\left\Vert x_{n+1}-x^{\ast }\right\Vert ^{2} &=&\left\langle
P_{C}t_{n}-x^{\ast },x_{n+1}-x^{\ast }\right\rangle  \notag \\
&=&\left\langle P_{C}t_{n}-t_{n},x_{n+1}-x^{\ast }\right\rangle
+\left\langle t_{n}-x^{\ast },x_{n+1}-x^{\ast }\right\rangle .  \label{A29}
\end{eqnarray}%
Also, from the property (3) of the metric projection $P_{C}$, since it
satisfies the inequality%
\begin{equation*}
\left\langle x-P_{C}x,y-P_{C}x\right\rangle \leq 0\text{, }\forall x\in H%
\text{, }y\in C,
\end{equation*}
we have%
\begin{eqnarray*}
\left\Vert x_{n+1}-x^{\ast }\right\Vert ^{2} &\leq &\left\langle
t_{n}-x^{\ast },x_{n+1}-x^{\ast }\right\rangle \\
&=&\left\langle \alpha _{n}\rho Vx_{n}+\left( I-\alpha _{n}\mu F\right)
T^{n}y_{n}-x^{\ast },x_{n+1}-x^{\ast }\right\rangle \\
&=&\left\langle \alpha _{n}\left( \rho Vx_{n}-\mu Fx^{\ast }\right) +\left(
I-\alpha _{n}\mu F\right) T^{n}y_{n}\right. \\
&&\left. -\left( I-\alpha _{n}\mu F\right) T^{n}x^{\ast },x_{n+1}-x^{\ast
}\right\rangle \\
&=&\alpha _{n}\rho \left\langle Vx_{n}-Vx^{\ast },x_{n+1}-x^{\ast
}\right\rangle +\alpha _{n}\left\langle \rho Vx^{\ast }-\mu Fx^{\ast
},x_{n+1}-x^{\ast }\right\rangle \\
&&+\left\langle \left( I-\alpha _{n}\mu F\right) T^{n}y_{n}-\left( I-\alpha
_{n}\mu F\right) T^{n}x^{\ast },x_{n+1}-x^{\ast }\right\rangle .
\end{eqnarray*}%
So, from Lemma \ref{cont}, we obtain%
\begin{eqnarray*}
\left\Vert x_{n+1}-x^{\ast }\right\Vert ^{2} &\leq &\alpha _{n}\rho \gamma
\left\Vert x_{n}-x^{\ast }\right\Vert \left\Vert x_{n+1}-x^{\ast
}\right\Vert +\alpha _{n}\left\langle \rho Vx^{\ast }-\mu Fx^{\ast
},x_{n+1}-x^{\ast }\right\rangle \\
&&+\left( 1-\alpha _{n}\nu \right) \left( \left\Vert y_{n}-x^{\ast
}\right\Vert +a_{n}\right) \left\Vert x_{n+1}-x^{\ast }\right\Vert \\
&\leq &\alpha _{n}\rho \gamma \left\Vert x_{n}-x^{\ast }\right\Vert
\left\Vert x_{n+1}-x^{\ast }\right\Vert +\alpha _{n}\left\langle \rho
Vx^{\ast }-\mu Fx^{\ast },x_{n+1}-x^{\ast }\right\rangle \\
&&+\left( 1-\alpha _{n}\nu \right) \left( \beta _{n}\left\Vert x_{n}-x^{\ast
}\right\Vert +\beta _{n}\left\Vert Sx^{\ast }-x^{\ast }\right\Vert \right. \\
&&+\left. \left( 1-\beta _{n}\right) \left\Vert z_{n}-x^{\ast }\right\Vert
+a_{n}\right) \left\Vert x_{n+1}-x^{\ast }\right\Vert \\
&\leq &\alpha _{n}\rho \gamma \left\Vert x_{n}-x^{\ast }\right\Vert
\left\Vert x_{n+1}-x^{\ast }\right\Vert +\alpha _{n}\left\langle \rho
Vx^{\ast }-\mu Fx^{\ast },x_{n+1}-x^{\ast }\right\rangle \\
&&+\left( 1-\alpha _{n}\nu \right) \left( \beta _{n}\left\Vert x_{n}-x^{\ast
}\right\Vert +\beta _{n}\left\Vert Sx^{\ast }-x^{\ast }\right\Vert \right. \\
&&+\left. \left( 1-\beta _{n}\right) \left\Vert x_{n}-x^{\ast }\right\Vert
+a_{n}\right) \left\Vert x_{n+1}-x^{\ast }\right\Vert \\
&\leq &\left( 1-\alpha _{n}\left( \nu -\rho \gamma \right) \right)
\left\Vert x_{n}-x^{\ast }\right\Vert \left\Vert x_{n+1}-x^{\ast }\right\Vert
\\
&&+\alpha _{n}\left\langle \rho Vx^{\ast }-\mu Fx^{\ast },x_{n+1}-x^{\ast
}\right\rangle \\
&&+\left( 1-\alpha _{n}\nu \right) \beta _{n}\left\Vert Sx^{\ast }-x^{\ast
}\right\Vert \left\Vert x_{n+1}-x^{\ast }\right\Vert \\
&&+\left( 1-\alpha _{n}\nu \right) a_{n}\left\Vert x_{n+1}-x^{\ast
}\right\Vert \\
&\leq &\frac{\left( 1-\alpha _{n}\left( \nu -\rho \gamma \right) \right) }{2}%
\left( \left\Vert x_{n}-x^{\ast }\right\Vert ^{2}+\left\Vert x_{n+1}-x^{\ast
}\right\Vert ^{2}\right) \\
&&+\alpha _{n}\left\langle \rho Vx^{\ast }-\mu Fx^{\ast },x_{n+1}-x^{\ast
}\right\rangle \\
&&+\left( 1-\alpha _{n}\nu \right) \beta _{n}\left\Vert Sx^{\ast }-x^{\ast
}\right\Vert \left\Vert x_{n+1}-x^{\ast }\right\Vert \\
&&+\left( 1-\alpha _{n}\nu \right) a_{n}\left\Vert x_{n+1}-x^{\ast
}\right\Vert .
\end{eqnarray*}%
The last inequality implies that%
\begin{eqnarray*}
\left\Vert x_{n+1}-x^{\ast }\right\Vert ^{2} &\leq &\frac{\left( 1-\alpha
_{n}\left( \nu -\rho \gamma \right) \right) }{\left( 1+\alpha _{n}\left( \nu
-\rho \gamma \right) \right) }\left\Vert x_{n}-x^{\ast }\right\Vert ^{2} \\
&&+\frac{2\alpha _{n}}{\left( 1+\alpha _{n}\left( \nu -\rho \gamma \right)
\right) }\left\langle \rho Vx^{\ast }-\mu Fx^{\ast },x_{n+1}-x^{\ast
}\right\rangle \\
&&+\frac{2\beta _{n}}{\left( 1+\alpha _{n}\left( -\rho \gamma \right)
\right) }\left\Vert Sx^{\ast }-x^{\ast }\right\Vert \left\Vert
x_{n+1}-x^{\ast }\right\Vert \\
&&+\frac{2a_{n}}{\left( 1+\alpha _{n}\left( \nu -\rho \gamma \right) \right) 
}\left\Vert x_{n+1}-x^{\ast }\right\Vert \\
&\leq &\left( 1-\alpha _{n}\left( \nu -\rho \gamma \right) \right)
\left\Vert x_{n}-x^{\ast }\right\Vert ^{2}+\alpha _{n}\left( \nu -\rho
\gamma \right) \theta _{n}
\end{eqnarray*}%
where%
\begin{equation*}
\theta _{n}=\frac{2}{\left( 1+\alpha _{n}\left( \nu -\rho \gamma \right)
\right) \left( \nu -\rho \gamma \right) }\left[ 
\begin{array}{c}
\left\langle \rho Vx^{\ast }-\mu Fx^{\ast },x_{n+1}-x^{\ast }\right\rangle +%
\frac{\beta _{n}}{\alpha _{n}}M_{3} \\ 
+\frac{a_{n}}{\alpha _{n}}\left\Vert x_{n+1}-x^{\ast }\right\Vert%
\end{array}%
\right] ,
\end{equation*}%
and 
\begin{equation*}
\sup_{n\geq 1}\left\{ \left\Vert Sx^{\ast }-x^{\ast }\right\Vert \left\Vert
x_{n+1}-x^{\ast }\right\Vert \right\} \leq M_{3}.
\end{equation*}%
Since $\frac{\beta _{n}}{\alpha _{n}}\rightarrow 0$ and $\frac{a_{n}}{\alpha
_{n}}\rightarrow 0,$ we obtain%
\begin{equation*}
\limsup_{n\rightarrow \infty }\theta _{n}\leq 0.
\end{equation*}%
So, it follows from Lemma \ref{Y} that the sequence $\left\{ x_{n}\right\} $
generated by (\ref{4}) converges strongly to $x^{\ast }\in \tciFourier $
which is the unique solution of variational inequality (\ref{1}).
\end{proof}

\begin{remark}
Under the suitable assumptions on parameters and operators in Theorem \ref{X}%
, we have the corresponding results of Yao et. al. \cite{YCL}, Marino and Xu 
\cite{marxu}, Ceng et. al. \cite{Ceng}, Wang and Xu \cite{WX}, Moudafi \cite%
{mou}, Xu \cite{xu1}, Tian \cite{tian} and Suzuki \cite{suzuki}. So, our
iterative process is a generalization form of many iteraitve processes
studied by above authors.
\end{remark}

Consequence of Theorem \ref{X}, we can obtain the following corollaries.

\begin{corollary}
\label{Y1} Let $C$ be a nonempty closed convex subset of a real Hilbert
space $H.$ Let $B:C\rightarrow H$ be $\theta $-inverse strongly monotone
mapping, $G:C\times C\rightarrow 
%TCIMACRO{\U{211d} }%
%BeginExpansion
\mathbb{R}
%EndExpansion
$ be a bifunction satisfying assumptions (A1)-(A4), $\varphi :C\rightarrow 
%TCIMACRO{\U{211d} }%
%BeginExpansion
\mathbb{R}
%EndExpansion
$ be a lower semicontinuous and convex function, $S:C\rightarrow H$ be a
nonexpansive mapping and $T$ be a demicontinuous nearly nonexpansive mapping
with the sequence $\left\{ a_{n}\right\} $ such that $\tciFourier
:=Fix\left( T\right) \cap \Omega \cap GMEP\left( G,\varphi ,B\right) \neq
\emptyset $. Let $V:C\rightarrow H$ be a $\gamma $-Lipschitzian mapping, $%
F:C\rightarrow H$ be a $L$-Lipschitzian and $\eta $-strongly monotone
operator such that these coefficients satisfy $0<\mu <\frac{2\eta }{L^{2}}$, 
$0\leq \rho \gamma <\nu $, where $\nu =1-\sqrt{1-\mu \left( 2\eta -\mu
L^{2}\right) }$. Assume that either (B1) or (B2) holds. For an arbitrarily
initial value $x_{1}\in C,$ consider the sequence $\left\{ x_{n}\right\} $
in $C$ generated by%
\begin{equation}
\left\{ 
\begin{array}{l}
G\left( u_{n},y\right) +\varphi \left( y\right) -\varphi \left( u_{n}\right)
+\left\langle Bx_{n},y-u_{n}\right\rangle +\frac{1}{r_{n}}\left\langle
y-u_{n},u_{n}-x_{n}\right\rangle \geq 0\text{, }\forall y\in C\text{\ \ \ \
\ \ } \\ 
y_{n}=P_{C}\left[ \beta _{n}Sx_{n}+\left( 1-\beta _{n}\right) u_{n}\right] ,%
\text{ } \\ 
x_{n+1}=P_{C}\left[ \alpha _{n}\rho Vx_{n}+\left( I-\alpha _{n}\mu F\right)
T^{n}y_{n}\right] ,\text{ }n\geq 1,%
\end{array}%
\right.  \label{S1}
\end{equation}%
where $\left\{ r_{n}\right\} \subset \left( 0,2\theta \right) $, $\left\{
\alpha _{n}\right\} $ and $\left\{ \beta _{n}\right\} $ are sequences in $%
\left[ 0,1\right] $ satisfying the conditions $(C1)$-$(C3)$ except the
condition $\lim_{n\rightarrow \infty }\frac{\left\vert \lambda _{n}-\lambda
_{n-1}\right\vert }{\alpha _{n}}=0$. Then, the sequence $\left\{
x_{n}\right\} $ generated by (\ref{S1})\ converges strongly to $x^{\ast }\in
\tciFourier $, where $x^{\ast }$\ is the unique solution of variational
inequality (\ref{1}).
\end{corollary}

\begin{corollary}
Let $C$ be a nonempty closed convex subset of a real Hilbert space $H.$ Let $%
B:C\rightarrow H$ be $\theta $-inverse strongly monotone mapping, $G:C\times
C\rightarrow 
%TCIMACRO{\U{211d} }%
%BeginExpansion
\mathbb{R}
%EndExpansion
$ be a bifunction satisfying assumptions (A1)-(A4), $\varphi :C\rightarrow 
%TCIMACRO{\U{211d} }%
%BeginExpansion
\mathbb{R}
%EndExpansion
$ be a lower semicontinuous and convex function and $T$ be a demicontinuous
nearly nonexpansive mapping with the sequence $\left\{ a_{n}\right\} $ such
that $\tciFourier :=Fix\left( T\right) \cap \Omega \cap GMEP\left( G,\varphi
,B\right) \neq \emptyset $. Let $V:C\rightarrow H$ be a $\gamma $%
-Lipschitzian mapping, $F:C\rightarrow H$ be a $L$-Lipschitzian and $\eta $%
-strongly monotone operator such that these coefficients satisfy $0<\mu <%
\frac{2\eta }{L^{2}}$, $0\leq \rho \gamma <\nu $, where $\nu =1-\sqrt{1-\mu
\left( 2\eta -\mu L^{2}\right) }$. Assume that either (B1) or (B2) holds.
For an arbitrarily initial value $x_{1}\in C,$ consider the sequence $%
\left\{ x_{n}\right\} $ in $C$ generated by%
\begin{equation}
\left\{ 
\begin{array}{l}
G\left( u_{n},y\right) +\varphi \left( y\right) -\varphi \left( u_{n}\right)
+\left\langle Bx_{n},y-u_{n}\right\rangle +\frac{1}{r_{n}}\left\langle
y-u_{n},u_{n}-x_{n}\right\rangle \geq 0\text{, }\forall y\in C\text{\ \ \ \
\ \ } \\ 
x_{n+1}=P_{C}\left[ \alpha _{n}\rho Vx_{n}+\left( I-\alpha _{n}\mu F\right)
T^{n}u_{n}\right] ,\text{ }n\geq 1,%
\end{array}%
\right.  \label{S2}
\end{equation}%
where $\left\{ r_{n}\right\} \subset \left( 0,2\theta \right) $, $\left\{
\alpha _{n}\right\} $ is a sequence in $\left[ 0,1\right] $ satisfying the
conditions $(C1)$-$(C3)$ except the conditions $\lim_{n\rightarrow \infty }%
\frac{\beta _{n}}{\alpha _{n}}=0$, $\frac{\left\vert \lambda _{n}-\lambda
_{n-1}\right\vert }{\alpha _{n}}=0$ and $\lim_{n\rightarrow \infty }\frac{%
\left\vert \beta _{n}-\beta _{n-1}\right\vert }{\alpha _{n}}=0$. Then, the
sequence $\left\{ x_{n}\right\} $ generated by (\ref{S2})\ converges
strongly to $x^{\ast }\in \tciFourier $, where $x^{\ast }$\ is the unique
solution of the variational inequality (\ref{1}).
\end{corollary}

\begin{corollary}
Let $C$ be a nonempty closed convex subset of a real Hilbert space $H$. Let $%
G:C\times C\rightarrow 
%TCIMACRO{\U{211d} }%
%BeginExpansion
\mathbb{R}
%EndExpansion
$ be a bifunction satisfying assumptions (A1)-(A4), $\varphi :C\rightarrow 
%TCIMACRO{\U{211d} }%
%BeginExpansion
\mathbb{R}
%EndExpansion
$ be a lower semicontinuous and convex function, $S:C\rightarrow H$ be a
nonexpansive mapping and $T$ be a demicontinuous nearly nonexpansive mapping
with the sequence $\left\{ a_{n}\right\} $ such that $\tciFourier
:=Fix\left( T\right) \cap MEP\left( G,\varphi \right) \neq \emptyset $. Let $%
V:C\rightarrow H$ be a $\gamma $-Lipschitzian mapping, $F:C\rightarrow H$ be
a $L$-Lipschitzian and $\eta $-strongly monotone operator such that these
coefficients satisfy $0<\mu <\frac{2\eta }{L^{2}}$, $0\leq \rho \gamma <\nu $%
, where $\nu =1-\sqrt{1-\mu \left( 2\eta -\mu L^{2}\right) }$. Assume that
either (B1) or (B2) holds. For an arbitrarily initial value $x_{1},$ define
the sequence $\left\{ x_{n}\right\} $ in $C$ generated by%
\begin{equation}
\left\{ 
\begin{array}{l}
G\left( u_{n},y\right) +\varphi \left( y\right) -\varphi \left( u_{n}\right)
+\frac{1}{r_{n}}\left\langle y-u_{n},u_{n}-x_{n}\right\rangle \geq 0\text{, }%
\forall y\in C\text{\ \ \ \ \ \ } \\ 
y_{n}=P_{C}\left[ \beta _{n}Sx_{n}+\left( 1-\beta _{n}\right) u_{n}\right] ,%
\text{ } \\ 
x_{n+1}=P_{C}\left[ \alpha _{n}\rho Vx_{n}+\left( I-\alpha _{n}\mu F\right)
T^{n}y_{n}\right] ,\text{ }n\geq 1,%
\end{array}%
\right.  \label{S3}
\end{equation}%
where $\left\{ r_{n}\right\} \subset \left( 0,\infty \right) $, $\left\{
\alpha _{n}\right\} $ and $\left\{ \beta _{n}\right\} $ are sequences in $%
\left[ 0,1\right] $ satisfying the conditions $(C1)$-$(C3)$ except the
condition $\lim_{n\rightarrow \infty }\frac{\left\vert \lambda _{n}-\lambda
_{n-1}\right\vert }{\alpha _{n}}=0$. Then, the sequence $\left\{
x_{n}\right\} $ generated by (\ref{S3})\ converges strongly to $x^{\ast }\in
Fix\left( T\right) \cap MEP(G,\varphi )$, where $x^{\ast }$\ is the unique
solution of variational inequality (\ref{1}).
\end{corollary}

\begin{corollary}
\label{Z} Let $C$ be a nonempty closed convex subset of a real Hilbert space 
$H.$ Let $A,B:C\rightarrow H$ be $\alpha ,\theta $-inverse strongly monotone
mappings, respectively. $G:C\times C\rightarrow 
%TCIMACRO{\U{211d} }%
%BeginExpansion
\mathbb{R}
%EndExpansion
$ be a bifunction satisfying assumptions (A1)-(A4), $\varphi :C\rightarrow 
%TCIMACRO{\U{211d} }%
%BeginExpansion
\mathbb{R}
%EndExpansion
$ be a lower semicontinuous and convex function, $S:C\rightarrow H$ be a
nonexpansive mapping and $T$ be a nonexpansive mapping such that $%
\tciFourier :=Fix\left( T\right) \cap \Omega \cap GMEP\left( G,\varphi
,B\right) \neq \emptyset $. Let $V:C\rightarrow H$ be a $\gamma $%
-Lipschitzian mapping, $F:C\rightarrow H$ be a $L$-Lipschitzian and $\eta $%
-strongly monotone operator such that these coefficients satisfy $0<\mu <%
\frac{2\eta }{L^{2}}$, $0\leq \rho \gamma <\nu $, where $\nu =1-\sqrt{1-\mu
\left( 2\eta -\mu L^{2}\right) }$. Assume that either (B1) or (B2) holds.
For an arbitrarily initial value $x_{1}\in C,$ consider the sequence $%
\left\{ x_{n}\right\} $ in $C$ generated by (\ref{4}) where $\left\{ \lambda
_{n}\right\} \subset \left( 0,2\alpha \right) $, $\left\{ r_{n}\right\}
\subset \left( 0,2\theta \right) $, $\left\{ \alpha _{n}\right\} $ and $%
\left\{ \beta _{n}\right\} $ are sequences in $\left[ 0,1\right] $
satisfying the conditions $(C1)$-$(C3)$ of Theorem \ref{X} except the
condition $\lim_{n\rightarrow \infty }\frac{a_{n}}{\alpha _{n}}=0$. Then,
the sequence $\left\{ x_{n}\right\} $\ converges strongly to $x^{\ast }\in
\tciFourier $, where $x^{\ast }$\ is the unique solution of the variational
inequality (\ref{1}).
\end{corollary}

\begin{remark}
Our results can be reduced to some corresponding results in the following
ways:

\begin{enumerate}
\item In our iterative process (\ref{S2}), if we take the mapping $T$ as
nonexpansive, $G\left( x,y\right) =0,\varphi =0$ for all $x,y\in C$, $B=0$
and $r_{n}=1$ for all $n\geq 1$, then we derive the iterative process%
\begin{equation*}
x_{n+1}=P_{C}\left[ \alpha _{n}\rho Vx_{n}+\left( I-\alpha _{n}\mu F\right)
Tx_{n}\right] ,\text{ }n\geq 1,
\end{equation*}%
which is studied by Ceng et. al. \cite{Ceng}.\ So, our results extend the
corresponding results of many other authors.

\item In our iterative process (\ref{S3}), if we take $S$ as a nonexpansive
self mapping on $C$, $T$ as a nonexpansive mapping, then it is clear that
our iterative process generalizes the iterative process of Wang and Xu. \cite%
{WX}. Hence, Theorem \ref{X} generalizes the main result of Wang and Xu \cite%
[Theorem 3.1]{WX}. So, our results extend and improve the corresponding
results of \cite{MX1, tian}.

\item The problem of finding the solution of variational inequality (\ref{1}%
), is equivalent to finding the solutions of hierarchical fixed point
problem 
\begin{equation*}
\left\langle \left( I-S\right) x^{\ast },x^{\ast }-x\right\rangle \leq
0,\forall x\in \tciFourier ,
\end{equation*}%
where $S=$ $I-\left( \rho V-\mu F\right) .$
\end{enumerate}
\end{remark}

\end{document}